\documentclass[a4paper,reqno,12pt]{amsart}
\usepackage[english]{babel}
\usepackage[utf8]{inputenc}
\usepackage[T1]{fontenc}
\usepackage[nomath]{lmodern} % without nomath, lmodern messes up font size

\usepackage{amssymb,amsmath,amsthm,amsfonts}
\usepackage{amsrefs}
\usepackage[thmtools-compat]{keytheorems}
\usepackage{mathtools}
\usepackage{microtype}
\allowdisplaybreaks
\usepackage{enumitem} % item labels
\usepackage{wasysym}
\usepackage[dvipsnames]{xcolor}
\usepackage{graphicx} % for including images
\usepackage{tikz,tikz-cd,pgfplots}
\pgfplotsset{compat=newest}
\usepackage[breaklinks,pdfencoding=auto,psdextra]{hyperref}
\usepackage{comment}
\usepackage[normalem]{ulem} % for strikeout

\addto\extrasenglish{}
\addto\extrasenglish{}
\addto\extrasenglish{\def\equationautorefname~#1\null{\textnormal{(#1)}\null}}
\declaretheorem[name=Theorem,refname={Theorem},style=plain,numberwithin=section]{theorem}
\declaretheorem[name=Theorem,refname={Theorem},style=plain,numbered=no]{theorem*}
\declaretheorem[name=Proposition,refname={Proposition},style=plain,sibling=theorem]{proposition}
\declaretheorem[name=Proposition,refname={Proposition},style=plain,numbered=no]{proposition*}
\declaretheorem[name=Lemma,refname={Lemma},style=plain,sibling=theorem]{lemma}
\declaretheorem[name=Lemma,refname={Lemma},style=plain,numbered=no]{lemma*}
\declaretheorem[name=Definition,refname={Definition},style=definition,sibling=theorem]{definition}
\declaretheorem[name=Definition,refname={Definition},style=definition,numbered=no]{definition*}
\declaretheorem[name=Remark,refname={Remark},style=definition,sibling=theorem]{remark}
\declaretheorem[name=Remark,refname={Remark},style=remark,numbered=no]{remark*}
\declaretheorem[name=Example,refname={Example},style=definition,sibling=theorem]{example}
\declaretheorem[name=Example,refname={Example},style=definition,numbered=no]{example*}
\declaretheorem[name=Corollary,refname={Corollary},style=plain,sibling=theorem]{corollary}
\declaretheorem[name=Corollary,refname={Corollary},style=plain,numbered=no]{corollary*}
\declaretheorem[name=Conjecture,refname={Conjecture},style=plain,sibling=theorem]{conjecture}

\newcounter{chr}
\loop\ifnum\value{chr}<27
  \expandafter\edef\csname b\Alph{chr}\endcsname{\noexpand\mathbf{\Alph{chr}}}
  \expandafter\edef\csname c\Alph{chr}\endcsname{\noexpand\mathcal{\Alph{chr}}}
  \expandafter\edef\csname f\Alph{chr}\endcsname{\noexpand\mathfrak{\Alph{chr}}}
\addtocounter{chr}{1}
\repeat

\def\Bl{\operatorname{Bl}}
\def\codim{\operatorname{codim}}
\def\Id{\operatorname{Id}}
\def\Pic{\operatorname{Pic}}

\def\Sing{\operatorname{Sing}}
\def\supp{\operatorname{supp}}
\def\Sym{\operatorname{Sym}}
\def\td{\operatorname{td}}

\def\inner#1{{\left<{#1}\right>}}
\def\set#1{{\left\{{#1}\right\}}}
\def\setmid#1#2{{\left\{{#1}\;\middle|\;{#2}\right\}}}
\def\setminus{\smallsetminus}

\def\q{\mathfrak q}
\def\ch{\mathrm{ch}}
\def\KKK{\mathrm{K3}}
\def\Kum{\mathrm{Kum}}
\def\OG{\mathrm{OG}}
\def\RR{\mathrm{RR}}
\def\RW{\mathrm{RW}}
\def\SH{\mathrm{SH}}
\def\top{\mathrm{top}}

\def\sigmabar{{\overline\sigma}}
\def\HT{{\mathrm{HT}}}
\def\h{{\mathrm{H}}}
\def\so{\mathfrak{so}}

\def\longarrow#1#2{\mathchoice{#2}{#1}{#1}{#1}}
\def\to{\longarrow{\rightarrow}{\longrightarrow}}
\def\simto{\longarrow{\xrightarrow\sim}{\stackrel\sim\longrightarrow}}

\let\shortmapsto\mapsto
\def\mapsto{\longarrow{\shortmapsto}{\longmapsto}}

\title{Second Chern class and Fujiki constants of hyperkähler manifolds}
\author{Thorsten Beckmann}
\address{Max--Planck--Institut für Mathematik, Vivatsgasse 7, 53111 Bonn, Germany}
\email{\url{beckmann@math.uni-bonn.de}}
\author{Jieao Song}
\address{Université Paris Cité, CNRS, IMJ-PRG, F-75013 Paris, France}
\email{\url{songjieao@gmail.com}}
\date{\today}
\usepackage[margin=1in]{geometry}
\begin{document}
\begin{abstract}
    We study characteristic classes on hyperkähler manifolds with a view towards the Verbitsky component. The case of the second Chern class leads to a conditional upper bound on the second Betti number in terms of the Riemann--Roch polynomial, which is also valid for singular examples. 
    We discuss the general structure of characteristic classes and the Riemann--Roch polynomial on hyperkähler manifolds using among other things Rozansky--Witten theory. 
    %, and we explain how the positivity of generalized Fujiki constants for certain higher degree Chern characters is involved.%\thorsten{Is this version better?}\jieao{my point is that the bound is not conjectural, only the positivity of $C(\ch_4)$ is}
\end{abstract}
\maketitle
\section{Introduction}
In the study of smooth projective varieties with trivial canonical bundle, irreducible compact hyperkähler manifolds take up a prominent place, as showcased by Beauville's famous decomposition theorem \cite{BeauvilleChern0}. However, up to date the amount of examples of deformation types of such manifolds is scarce. It is therefore natural to study a priori topological restrictions that such varieties must obey. There are several results in this direction, for example \cites{Guan4dimHK, JiangRR, Salamon, KimLaza, SawonThesis, Sawon2021}.

Given an irreducible hyperkähler manifold $X$ of dimension $2n$, its second cohomology group $H^2(X,\bR)$ is equipped with the Beauville--Bogomolov--Fujiki form $q_X$. Moreover, the full cohomology ring $H^\ast(X,\bR)$ is naturally a module under the Looijenga--Lunts--Verbitsky (LLV) Lie algebra $\mathfrak{g}(X)_{\bR}$ \cites{LooijengaLunts, VerbitskyCohomologyHK, GKLRLLV}. This leads to a decomposition of $H^\ast(X,\bR)$ into irreducible representations. Arguably, the most important one is the Verbitsky component $\SH(X,\bR)\subset H^\ast(X,\bR)$, which is the subalgebra generated by $H^2(X,\bR)$. 

A natural question that arises is how much information this subalgebra encodes on the full cohomology. For example, one could ask which Chern classes of sheaves and, in particular, characteristic classes are contained inside the Verbitsky component. 

One case we consider here is that of the second Chern class $c_2\coloneqq c_2(X)\in H^4(X,\bR)$. Maybe a priori counter-intuitively, it is not always contained in the Verbitsky component, see for example \cite[Lemma~1.5]{MarkmanBBFormChar} for the case of the Hilbert scheme of $n$ points on a K3 surface with $n>3$. Note that $c_2$ lies in the Verbitsky component if and only if it is a multiple of the class $\q\in H^4(X,\bQ)$, the dual of the Beauville--Bogomolov--Fujiki form. 

We will answer completely the following question using the Riemann--Roch polynomial of $X$: ``When does $c_2$ lie completely inside the Verbitsky component of $X$?'' Recall that for a class $\alpha \in H^{4k}(X,\bR)$ which remains of type $(2k,2k)$ on all small deformations of $X$, there exists a number $C(\alpha)$, called the \emph{generalized Fujiki constant} of $\alpha$, such that
\begin{equation}
   \forall \beta \in H^2(X,\bR)\quad C(\alpha) \cdot q_X(\beta)^{n-k} = \int_X \alpha \cdot \beta^{2n-2k}.
\end{equation}
Let $\td$ be the Todd class of $X$ and let $\td_{2k}$ be its degree $2k$ part. The Riemann--Roch polynomial of $X$ is defined as \[
\begin{aligned}
\RR_X(q)\coloneqq\sum_{i=0}^n\frac{C(\td_{2n-2i})}{(2i)!}q^i
&=\frac{C(1)}{(2n)!}q^n+
\frac{C(\td_2)}{(2n-2)!}q^{n-1}+\cdots+\frac{C(\td_{2n})}1\\
&\eqqcolon A_0q^n+A_1q^{n-1}+A_2q^{n-2}+\cdots+A_n.
\end{aligned}
\]
The Hirzebruch--Riemann--Roch theorem, whence the name, together with the property of the generalized Fujiki constants assert that this polynomial satisfies
\[
\RR_X(q_X(c_1(L))) = \chi(X,L)
\]
for all line bundles $L \in \Pic(X)$. 
In particular, we have $A_n=n+1$.%See Section~\ref{sec:The_Inequality} for details.

The following is the main result which, additionally, yields an upper bound on the second Betti number $b_2(X)$ under some conditions. 
\begin{theorem}
\label{thm:bound_b2_c2_Verbitsky}
Let $X$ be a hyperkähler manifold of dimension $2n\ge 4$ with second Betti number
$b_2(X)$ and consider its Riemann--Roch polynomial
\[\RR_X(q)=A_0q^n+A_1q^{n-1}+A_2 q^{n-2}+\cdots.\]
\begin{enumerate}
\item The following conditions are equivalent.
\begin{itemize}
\item The first three coefficients satisfy the condition
\begin{equation}
\label{eq:hypothesis}
% \tag{$*$}
2nA_0A_2<(n-1)A_1^2,
\end{equation}
\item The inequality $C(c_2^2)>2C(c_4)$ holds, or equivalently, $C(\ch_4)>0$.
% \begin{equation}
% \label{eq:ineq_introduction_ch_4}
% \tag{$*'$}
% C(c_2^2)>2C(c_4).
% \end{equation}
\end{itemize}
Moreover, \eqref{eq:hypothesis} holds if the polynomial $\RR_X(q)$ has $n$ distinct real roots.
\item If \eqref{eq:hypothesis} holds, then we have the following inequality for $b_2(X)$
\begin{equation}
\label{eq:bound}
b_2(X) \le \frac1{1-\dfrac{2nA_0A_2}{(n-1)A_1^2}}-(2n-2)=9-2n + \frac{10C(c_4)}{C(c_2^2)-2C(c_4)}.
\end{equation}
\item The second Chern class $c_2$ lies in the Verbitsky component $\Sym^2 H^2(X,\bR)$ if and only if both the condition~\eqref{eq:hypothesis} holds and $b_2(X)$ attains the bound in~\eqref{eq:bound}.
\end{enumerate}
\end{theorem}
% The theorem can also be phrased using generalized Fujiki constants of (products of) Chern classes. 
% Namely, inequality~\eqref{eq:hypothesis} is equivalent to the condition that the generalized Fujiki constant $C(\ch_4)$ is positive or, expressed differently, that
% \begin{equation}
% \label{eq:ineq_introduction_ch_4}
% C(c_2^2)>2C(c_4).
% \end{equation}
% This is satisfied if the polynomial $\RR_X(q)$ has $n$ distinct real roots, see Remark~\ref{rmk:RR_split}.
Note that when $n=2$, we have $C(\ch_4)>0$ by \cite[Lemma~5.8]{OberdieckSongVoisin}, hence \eqref{eq:hypothesis} always holds.
% Writing $C(c_2^2) = \mu C(c_4)$ condition \eqref{eq:hypothesis} is equivalent to $\mu >2$ and the bound \eqref{eq:bound} becomes
% \begin{equation}
% b_2(X)\le 9-2n+\frac{10}{\mu-2}.    
% \end{equation}

We show in Corollary~\ref{cor:td_2n-2_verbitskycomponent} that the above conditions are also necessary and sufficient for $\td^{1/2}_{2n-2}\in H^{4n-4}(X,\bR)$, i.e.\ the degree $2n-2$ component of the square root of the Todd class, to be contained in the Verbitsky component. 

Among known smooth hyperkähler manifolds, there are only two types of Riemann--Roch polynomials: the $\mathrm K3^{[n]}$-type and the $\Kum_n$-type ($\OG_6$ and $\OG_{10}$ fall into these two types, see~\cite{OrtizRRPolynomials}). On the other hand, Theorem~\ref{thm:bound_b2_c2_Verbitsky} can be generalized to singular symplectic varieties of dimension 4 and this gives rise to many more examples. We check that the inequality $C(\ch_4)>0$ is satisfied for all known smooth examples, as well as for many singular examples, in Sections~\ref{sec:The_Inequality} and \ref{sec:orbifold_examples} respectively. 

%\jieao{I don't think we have all the known singular examples here: e.g. the dual $\Kum_2$}

In Section~\ref{sec:generl_fujiki_known_smooth}, we give an account of all generalized Fujiki constants for the known examples of smooth hyperkähler manifolds. In particular, we prove that when $X$ is of $\OG_6$ or $\OG_{10}$-deformation type, all Chern classes $c_{2i}$ satisfy
\[
c_{2i} \in \SH(X,\bR)
\]
and, thus, all characteristic classes of $X$ lie in the Verbitsky component.
This easily leads to the determination of the generalized Fujiki constants for all characteristic classes on these manifolds. 

In the final section, we further discuss generalized Fujiki constants and Riemann--Roch polynomials using Rozansky--Witten theory. We present a conceptual proof for the fact that the polynomial
\begin{equation*}
    \RR_{X,1/2}(q) \coloneqq \sum_{i=0}^n \frac{C(\td^{1/2}_{2n-2i})}{(2i)!}q^i
\end{equation*}
factorizes as an $n$-th power using the Wheeling Theorem and discuss how this method could be used in general to analyze the Riemann--Roch polynomial. This leads to conjectural relations between the generalized Fujiki constants. We mention here the degree four case which yields a precise value of $C(\ch_4)$. For another instance of these conjectural relations, see Conjecture~\ref{conj:Generalized_Fujiki_ch_8_via_RW}.% There could be other relations too, so I just leave this sentence out.
%Expressing the generalized Fujiki constant $C(\ch_4)$ using a certain trivalent graph leads to the following conjecture for the precise value of $C(\ch_4)$. 
\begin{conjecture}
\label{conj:Mercedes_graph}
Let $X$ be a hyperkähler manifold of dimension $2n\ge 4$. We have
\[
\frac{C(\ch_4)}{C(1)} = \frac{5(n+1)}{(2n-1)(2n-3)}.
\]
\end{conjecture}
Note that, in particular, Conjecture~\ref{conj:Mercedes_graph} would imply the positivity of $C(\ch_4)$. 
We prove in Proposition~\ref{prop:conj_RR_poly_implies_conj_C(ch_4)} that the conjecture holds true if the Riemann--Roch polynomial satisfies certain expectations on its shape such as \cite[Conjecture~1.3 (3)]{JiangRR} or Conjecture~\ref{conj:general_form_RR_polynomial}. We present a possible strategy towards proving these conjectures. 

We want to remark that we expect the positivity of $C(\ch_4)$ to hold true pointwise on the level of forms for the right representative of $\ch_4$ and therefore be of local nature. In contrast, Conjecture~\ref{conj:Mercedes_graph} is of global nature. The distinction between these two expectations will occur frequently in the paper. 

If proven true, Conjecture~\ref{conj:Mercedes_graph} would imply that for hyperkähler fourfolds there are exactly two possible sets of values that the generalized Fujiki constants can take, see Proposition~\ref{prop:possible_Fujiki_dim4}. As a consequence, we obtain the following. 
\begin{corollary}
\label{cor:possible_Hodge_numbers_via_conjecture}
Assuming Conjecture~\ref{conj:Mercedes_graph} in dimension $4$, the Betti numbers of a hyperkähler fourfold are one of the following: \begin{itemize}
    \item $b_2(X)=5, b_3(X)=0, b_4(X)=96$;
    \item $b_2(X)=6, b_3(X)=4, b_4(X)=102$;
    \item $b_2(X)=7, b_3(X)=8, b_4(X)=108$;
    \item $b_2(X)=23, b_3(X)=0, b_4(X)=276$.
\end{itemize}
\end{corollary}
Hence, Conjecture~\ref{conj:Mercedes_graph} would reduce the number of possible Hodge diamonds and LLV decompositions of hyperkähler fourfolds to four. The two known cases are the ones where $c_2$ lies in the Verbitsky component \cite{GKLRLLV}. In the case $b_2(X)=7$, there are 80 trivial representations of the LLV algebra in $H^{2,2}$, whereas there are 81 trivial representations when the second Betti number is smaller than seven. 

In the recent work~\cite[Theorem~9.3]{DHMV}, the authors obtained a similar result under a different assumption. We remark that the condition in our Conjecture~\ref{conj:Mercedes_graph} is stronger but makes no explicit assumption on the lattice $H^2(X,\bZ)$. It focuses only on numerical properties of the Riemann--Roch polynomial.
\subsection*{Relation to other work} While working on further results related to the topic of the paper we
learned about the recent preprint of Justin Sawon
\cite{Sawon2021} who independently obtained the same bound on the second Betti number as
in Theorem~\ref{thm:bound_b2_c2_Verbitsky}. The pointwise conjectural
relations in Section~\ref{sec:further_discussions} have a similar flavor as the ones in \cite[Section~2]{Sawon2021}. 

\subsection*{Acknowledgements} We are grateful to our supervisors Olivier Debarre and Daniel Huybrechts for their support and encouragement, and we thank Georg Oberdieck and \'Angel David R\'ios Ortiz for helpful conversations, as well as Lie Fu, Grégoire Menet, and Justin Sawon for useful comments.
Furthermore, we express our sincere gratitude to the anonymous referees, whose suggestions allowed us to greatly improve the presentation of the paper.
The first named author is funded by the IMPRS program of the Max--Planck--Society. 
\section{The inequality}
\label{sec:The_Inequality}
We prove Theorem~\ref{thm:bound_b2_c2_Verbitsky} in this section.
Let $X$ be a hyperkähler manifold of complex dimension $2n\ge 4$.
We first recall the following result by Fujiki \cite{FujikiCohomology} and Huybrechts~\cite{HuybrechtsHKBasicResults}. 
\begin{theorem}[Fujiki, Huybrechts]
\label{thm:fujiki}
Let $\alpha\in H^{4k}(X,\bR)$ be a class that remains of type $(2k,2k)$ on all small
deformations of $X$ (for example, all characteristic classes satisfy this
condition). Then there exists a constant $C(\alpha)\in\bR$, called the {\em
generalized Fujiki constant of $\alpha$}, such that
\[
\forall \beta\in H^2(X,\bR)\quad C(\alpha)\cdot
q_X(\beta)^{n-k} = \int_X \alpha\cdot \beta^{2n-2k}.
\]
\end{theorem}
\begin{remark}
The term \emph{Fujiki constant} is reserved for the value $C(1)=C(1_X)$.
There is also the notion of \emph{small Fujiki constant} $c_X$: it differs from $C(1)$ by a constant multiple
\[
C(1) = \frac{(2n)!}{2^nn!}\cdot c_X=(2n-1)!!\cdot c_X.
\]
For example, it is known that $c_{\mathrm K3^{[n]}}=1$ and $c_{\Kum_n}=n+1$.
\end{remark}
Denote by $\q\in \Sym^2H^2(X,\bR)$ the dual of the Beauville--Bogomolov--Fujiki
form, and by $\SH(X,\bR)\subset H^\ast(X,\bR)$ the Verbitsky
component, which is the subalgebra generated by $H^2(X,\bR)$. The key step to Theorem~\ref{thm:bound_b2_c2_Verbitsky} is the following result, which follows the same idea as \cite[Lemma~3]{Guan4dimHK}.
\begin{proposition}
\label{prop:fundamental_inequality}
We have the following inequality
\begin{equation}
\label{eq:inequality_c2_square}
C(c_2^2)\ge \frac{C(c_2)^2}{C(\q)^2}C(\q^2),
\end{equation}
where equality holds if and only if $c_2\in \Sym^2H^2(X,\bR)$.
\end{proposition}
\begin{proof}
Note that, up to scalar, $\q$ is the unique class in $\Sym^2 H^2(X,\bR)$ that remains of type $(2,2)$ on all small deformations of $X$ (this follows for example from the local Torelli theorem \cite[Théorème~5]{BeauvilleChern0}).
Therefore, we may write
\[
c_2=a\q+z\quad\text{where }a\in\bR,z\in \SH(X,\bR)^\perp.
\]
In other words, we project $c_2$ orthogonally to the Verbitsky component and
let $a\q$ be its image. Then we have
\[
C(c_2)=C(a\q),\quad\text{so } a=\frac{C(c_2)}{C(\q)}.
\]
Now we consider the square $c_2^2=a^2\q^2+2a\q z+z^2\in H^8(X,\bR)$. Since the class $z$ is
in $\SH(X,\bR)^\perp$, it is orthogonal to the image of
$\Sym^{2n-2}H^2(X,\bR)$, so the class $\q z$ is orthogonal to the image of
$\Sym^{2n-4}H^2(X,\bR)$ and also lies in
$\SH(X,\bR)^\perp$.

On the other hand, for any Kähler class $\omega\in H^2(X,\bR)$, since $z$ lies in $\SH(X,\bR)^\perp$, the class $z\cdot\omega^{2n-3}\in H^{4n-2}(X,\bR)$ is orthogonal to the entire $H^2(X,\bR)$ hence must vanish. So the class $z$ is primitive of type $(2,2)$ with respect to all Kähler classes on $X$.
By the
Hodge--Riemann bilinear relations, for a Kähler class $\omega\in H^2(X,\bR)$ we have
\[
\int_X z^2\cdot \omega^{2n-4}\ge 0,\quad\text{hence }C(z^2)\ge0,
\]
where equality holds if and only if $z=0$, that is, if $c_2\in
\Sym^2H^2(X,\bR)$.
In other words, the projection of $z^2$ to the Verbitsky component is
non-trivial, unless $z$ is itself trivial. Therefore we obtain the desired inequality
\[
C(c_2^2)=a^2C(\q^2)+C(z^2)\ge a^2C(\q^2)=
\frac{C(c_2)^2}{C(\q)^2}C(\q^2),
\]
where equality holds if and only if $c_2\in \Sym^2H^2(X,\bR)$.
\end{proof}

We now study the values of the various generalized Fujiki constants that appear
in~\eqref{eq:inequality_c2_square}.

\begin{proposition}
\label{prop:fujiki_of_q}
Let $X$ be a hyperkähler manifold of dimension $2n$ with second Betti number
$b\coloneqq b_2(X)$. For any $\alpha\in H^{4k}(X,\bR)$ that is of type $(2k,2k)$
on all small deformations of $X$, we have
\[
C(\q\cdot\alpha) = \frac{b+2n-2k-2}{2n-2k-1}C(\alpha).
\]
In particular, we get
\[
C(\q^{k})=\frac{b+2n-2k}{1+2n-2k}C(\q^{k-1})
=\prod_{i=1}^{k}\frac{b+2n-2i}{1+2n-2i}\cdot C(1).
\]
\end{proposition}
\begin{proof}
Take a basis $(e_1,\dots,e_b)$ of $H^2(X,\bR)$ such that
\[\q=e_1^2+e_2^2+e_3^2-e_4^2-\cdots-e_b^2.\]
Writing $s_i\coloneqq q_X(e_i)\in\set{\pm1}$ so $\q=\sum_i s_i e_i^2$, we have
\[
\begin{aligned}
C(\q\cdot \alpha)=\int_X \q\cdot \alpha\cdot e_1^{2n-2k-2}
&=\int_X\sum_i s_ie_i^2\cdot\alpha\cdot e_1^{2n-2k-2}\\
&=C(\alpha)+\sum_{i>1}s_i\int_X \alpha \cdot e_1^{2n-2k-2}e_i^2.
\end{aligned}
\]
For each term $e_1^{2n-2k-2}e_i^2$, consider the function
\[
t\mapsto \int_X \alpha\cdot (e_1+te_i)^{2n-2k}=
C(\alpha)\cdot(1+t^2s_i)^{n-k},
\]
which is a polynomial in $t$. Comparing the coefficients of $t^2$, we get
\[
\binom{2n-2k}2\int_X\alpha\cdot e_1^{2n-2k-2}e_i^2=C(\alpha)\cdot (n-k)s_i.
\]
So we have
\[
C(\q\cdot\alpha) = C(\alpha)+\sum_{i>1}s_i \frac{C(\alpha)s_i}{2n-2k-1}
= C(\alpha)+(b-1)\frac{C(\alpha)}{2n-2k-1}
= \frac{b+2n-2k-2}{2n-2k-1}C(\alpha),
\]
where we used the fact that $s_i^2=1$.
\end{proof}

On the other hand, we have the following result by Nieper-Wißkirchen~\cite{NieperWissHRRonIHS}, which generalizes the work of Hitchin--Sawon~\cite{HitchinSawon}. In particular, it produces linear relations among certain generalized Fujiki constants. We will present a short proof of this theorem in Section~\ref{sec:proof_RR_td_1/2} where we see conceptually where the factorization comes from.
\begin{theorem}
\label{thm:RR_td_1/2}
Let $X$ be a hyperkähler manifold of dimension $2n$. Consider the following polynomial
\[
\begin{aligned}
\RR_{X,1/2}(q)&\coloneqq\sum_{i=0}^n\frac{C(\td^{1/2}_{2n-2i})}{(2i)!}q^i\\
&=\frac{C(1)}{(2n)!}q^n+
\frac{C(\frac1{24}c_2)}{(2n-2)!}q^{n-1}+
\frac{C(\frac7{5760}c_2^2-\frac1{1440}c_4)}{(2n-4)!}q^{n-2}+\cdots+\frac{C(\td^{1/2}_{2n})}1.
\end{aligned}
\]
There exists a constant $r_X$ such that this polynomial factorizes as
\[
\RR_{X,1/2}(q)=C(\td^{1/2}_{2n})\left(1+\frac1{2r_X} q\right)^n.
\]
\end{theorem}
In \cite[Section~3]{BeckmannExtendedIntegral}, by comparing the first two coefficients, it is shown that
\[
r_X=\frac{(2n-1)C(c_2)}{24C(1)}=\frac{(2n-1)2^nn!C(c_2)}{24(2n)!c_X},
\]
and
\[
C(\td^{1/2}_{2n})=\frac{C(1)(2r_X)^n}{(2n)!}=c_X\frac{r_X^n}{n!},
\]
where $c_X$ is the small Fujiki constant.
By comparing the third coefficients, we get the following relation.
\begin{corollary}
\label{cor:relation_td1/2}
Let $X$ be a hyperkähler manifold of dimension $2n\ge 4$. Then
\[
7C(c_2^2)-4C(c_4) = \frac{5(2n-1)C(c_2)^2}{(2n-3)C(1)}.
\tag*{\qed}
\]
\end{corollary}
% Combining Corollary~\ref{cor:relation_td1/2} and \eqref{eq:inequality_fujiki} we obtain
% \begin{equation}
% \label{eq:ineq_c2_2_and_c4}
%     C(c_2^2) \geq \frac{b_2(X) + 2n -4}{5(b_2(X) + 2n-2)}(7C(c_2^2)-4C(c_4))
% \end{equation}
% which is equivalent to 
% \begin{equation}
% \label{eq:ineq_c2_2_and_c4_multiplied}
%     (C(c_2^2)-2C(c_4))(b_2(X)+2n-9) \leq 10C(c_4).
% \end{equation}
% The last missing ingredient for proving Theorem~\ref{thm:bound_b2_c2_Verbitsky} is the connection of the above with the Riemann--Roch polynomial. 
\begin{corollary}
\label{cor:first_four_C}
All generalized Fujiki constants for characteristic classes of degree~$\le
4$ are determined by the Riemann--Roch polynomial, or more precisely, by its
first three coefficients
\[
\begin{aligned}
\RR_X(q)=\sum_{i=0}^n\frac{C(\td_{2n-2i})}{(2i)!}q^i
&=\frac{C(1)}{(2n)!}q^n+
\frac{C(\frac1{12}c_2)}{(2n-2)!}q^{n-1}+
\frac{C(\frac1{240}c_2^2-\frac1{720}c_4)}{(2n-4)!}q^{n-2}+\cdots\\
&=A_0q^n+A_1q^{n-1}+A_2q^{n-2}+\cdots
\end{aligned}
\]
Namely, we have
\[
\begin{gathered}
C(1) = (2n)!A_0,\quad C(c_2) = 12(2n-2)!A_1,\\
C(c_2^2) = 144(2n-4)!\left(4A_2-\frac{(n-1)A_1^2}{nA_0}\right),\\
C(c_4) = 144(2n-4)!\left(7A_2-\frac{3(n-1)A_1^2}{nA_0}\right).
\end{gathered}
\]
\end{corollary}
\begin{proof}
Clearly $C(1)$ and $C(c_2)$ appear as coefficients of the Riemann--Roch
polynomial.
For $C(c_2^2)$ and $C(c_4)$, we already have one linear relation
\[
7C(c_2^2)-4C(c_4) = \frac{5(2n-1)C(c_2)^2}{(2n-3)C(1)}=
720(2n-4)!\frac{(n-1)A_1^2}{nA_0}.
\]
The third coefficient gives another one
\[
3C(c_2^2)-C(c_4) = 720(2n-4)!A_2,
\]
which allows us to uniquely determine their values.
Hence we get all four generalized Fujiki constants of degree $\le 4$.
\end{proof}
We note that other coefficients of the Riemann--Roch polynomial also provide interesting linear relations for generalized Fujiki constants for characteristic classes of degree $>4$ but do not determine them completely, since the number of partitions grows very quickly (see also Section~\ref{sec:generl_fujiki_known_smooth} for examples sharing the same Riemann--Roch polynomial).
\begin{proof}[Proof of Theorem~\ref{thm:bound_b2_c2_Verbitsky}]
Using Corollary~\ref{cor:first_four_C}, we can easily deduce that the condition~\eqref{eq:hypothesis} is equivalent to $C(c_2^2)>2 C(c_4)$.
If we suppose that the Riemann--Roch polynomial factorizes as a product of linear factors
\[
\RR_X(q)=A_0\prod_i(q+\lambda_i),
\]
so $\frac{A_1}{A_0}=\sum_i\lambda_i$ and $\frac{A_2}{A_0}=\sum_{i<j}\lambda_i\lambda_j$, then the condition~\eqref{eq:hypothesis} becomes
\[
2n\sum_{i<j}\lambda_i\lambda_j < (n-1)\big(\sum_i\lambda_i\big)^2 \text{, or equivalently, } \big(\sum_i\lambda_i\big)^2 < n\cdot\sum_i\lambda_i^2.
\]
Thus it holds by Cauchy--Schwarz as soon as we assume that $\lambda_i$ are not all equal.%
\footnote{
Note that in this case, the inequality~\eqref{eq:bound} becomes
$
b_2(X)\le
\frac{n-1}{\frac{n\cdot\sum_i\lambda_i^2}{(\sum_i\lambda_i)^2}-1}-(2n-2),
$
which is homogeneous with respect to the $\lambda_i$ and measures in a certain sense the dispersion of the roots.}

By Proposition~\ref{prop:fundamental_inequality}, we have the inequality \eqref{eq:inequality_c2_square}, and moreover $c_2 \in \SH^4(X,\bR)$ if and only if it is in fact an equality.
Using the description in Proposition~\ref{prop:fujiki_of_q} to replace $C(\q)$ and $C(\q^2)$ in~\eqref{eq:inequality_c2_square}, we get
\begin{equation}
\label{eq:inequality_fujiki}
C(c_2^2)\ge
\frac{(2n-1)(b_2(X)+2n-4)C(c_2)^2}{(2n-3)(b_2(X)+2n-2)C(1)}.
\end{equation}
Note in particular that $C(c_2^2)>0$.
Applying the identity in Corollary~\ref{cor:relation_td1/2}, we obtain
\[
% \label{eq:ineq_c2_2_and_c4}
C(c_2^2) \geq \frac{b_2(X) + 2n -4}{5(b_2(X) + 2n-2)}(7C(c_2^2)-4C(c_4))
\]
which is equivalent to 
\begin{equation}
\label{eq:ineq_c2_2_and_c4_multiplied}
    (C(c_2^2)-2C(c_4))(b_2(X)+2n-9) \leq 10C(c_4).
\end{equation}
Therefore, if we assume that $C(c_2^2)>2C(c_4)$,
we obtain the following inequality as in \eqref{eq:bound}
\[
b_2(X)\le 9-2n+\frac{10C(c_4)}{C(c_2^2)-2C(c_4)}.
\]
Using again Corollary~\ref{cor:first_four_C}, it is straightforward to obtain the equivalent expression in terms of the coefficients $A_i$.

Finally, when the equality holds in~\eqref{eq:ineq_c2_2_and_c4_multiplied}, we have
\[
(C(c_2^2)-2C(c_4))(b_2(X)+2n-2) = 7C(c_2^2)-4C(c_4)
\]
which is positive by Corollary~\ref{cor:relation_td1/2}.
Therefore $C(c_2^2)>2C(c_4)$ must hold as well.
\end{proof}

The condition for $c_2$ to be contained inside the Verbitsky component actually also gives an equivalent condition for $\td^{1/2}_{2n-2}$ to lie inside $\SH(X,\bR)$, by the following result.
\begin{proposition}
For a hyperkähler manifold $X$ of dimension $2n$, we have $\td^{1/2}_{2k}\in \SH(X,\bR)$ if and only if $\td^{1/2}_{2n-2k} \in \SH(X,\bR)$. Moreover, $\td^{1/2}_{2k}\in \SH(X,\bR)$ implies $\td^{1/2}_{2k'}\in \SH(X,\bR)$ for $k'<k<n$.
\end{proposition}
\begin{proof}
For a class $\alpha\in H^2(X,\bC)$, denote by $e_\alpha$ the operator $x\mapsto x\wedge\alpha$ in the LLV algebra $\mathfrak g(X)_\bC$.
Define $h_p$ to be the holomorphic grading operator that acts on $H^{p,q}(X)$ as $(n-p)\Id$ (which is denoted by $\Pi$ in~\cite{JiangRR}), and similarly the antiholomorphic grading operator $h_q$ which acts on $H^{p,q}(X)$ as $(n-q)\Id$. 
Recall that for the class $\sigma$ of a symplectic form, the operator $e_\sigma$ has the Lefschetz property with respect to the grading given by $h_p$: there exists a dual Lefschetz operator $\Lambda_\sigma \in \mathfrak{g}(X)_{\bC}$, such that together with the operator~$h_p$, we get an $\mathfrak{sl}_2$-triple $(e_\sigma, h_p, \Lambda_\sigma)$ in the LLV algebra. The same result holds if we consider $e_\sigmabar$ and $h_q$.

Jiang \cite[Corollary~3.19]{JiangRR} showed that there exists a constant $r_\sigma\in \bR_{>0}$ such that
\begin{equation}
\label{eq:Jiang_td1/2}
    \Lambda_\sigma(\td_{2k}^{1/2}) =r_\sigma \td^{1/2}_{2k-2}\wedge \sigmabar.
\end{equation}
Furthermore, the operators $e_{\sigmabar}$ and $\Lambda_\sigma$ commute for degree reasons (see also \cite[Lemma~2.9]{JiangRR}). Applying~\eqref{eq:Jiang_td1/2} repeatedly, we see that the following holds for all $k<n/2$
\begin{equation}
\label{eq:proof_prop210}
    \Lambda_\sigma^{n-2k}(\td^{1/2}_{2n-2k})= r_\sigma^{n-2k}\td_{2k}^{1/2}\wedge \sigmabar^{n-2k}.
\end{equation}

On the other hand, Fujiki \cite{FujikiCohomology} showed that the operators $e_{\sigmabar}$ and $\Lambda_\sigma$ yield isomorphisms
\[
e_{\sigmabar}^{s} \colon H^{l,n-s}(X) \simto H^{l,n+s}(X) , \quad \Lambda_\sigma^s \colon H^{n+s,l}(X) \simto H^{n-s,l}(X).
\]
Moreover, these isomorphisms are compatible with the decomposition of $H^\ast(X,\bC)$ into irreducible $\mathfrak{g}(X)_{\bC}$-representations, i.e.\ for each irreducible representation $V \subset H^\ast(X,\bC)$, the isomorphism $e_{\sigmabar}^s$ restricts to an isomorphism
\[
e_{\sigmabar}^{s} \colon H^{l,n-s}(X) \cap V \simto H^{l,n+s}(X) \cap V,
\]
and similar for $\Lambda_\sigma^s$. 
Combining this with \eqref{eq:proof_prop210} yields the first assertion. 
The second statement also follows from \eqref{eq:proof_prop210} using the same line of arguments. 
\end{proof}
\begin{corollary}
\label{cor:td_2n-2_verbitskycomponent}
For a hyperkähler manifold $X$ of dimension $2n$, the class $\td^{1/2}_{2n-2}$ lies in the Verbitsky component if and only if the condition~\eqref{eq:hypothesis} is satisfied and the equality in~\eqref{eq:bound} holds.  
\end{corollary}

\bigskip

We now examine the bound \eqref{eq:bound} for the known deformation types of smooth hyperkähler manifolds. There are only two known types of Riemann--Roch polynomials
\[
\RR_{\mathrm K3^{[n]}}(q) = \binom{q/2+n+1}n,\quad
\RR_{\Kum_n}(q) = (n+1)\binom{q/2+n}n,
\]
see \cite[Lemma~5.1]{EGL} and \cite[Lemma~5.2]{NieperWissHRRonIHS}. Ríos Ortiz showed that O'Grady's sporadic examples satisfy $\RR_{\OG_{10}}(q)=\RR_{\mathrm K3^{[5]}}(q)$ and $\RR_{\OG_{6}}(q)=\RR_{\Kum_3}(q)$ in~\cite{OrtizRRPolynomials}. %\jieao{Should Ángel be cited as Ortiz or Ríos or Ríos Ortiz? Should we ask him?}

\begin{example}[$\mathrm K3^{[n]}$-type]
We compute the first three coefficients
\[
\begin{aligned}
\RR_{\mathrm K3^{[n]}}(q) &= \binom{q/2+n+1}n\\
&=\frac1{2^nn!}q^n +\frac{n+3}{2^n(n-1)!}q^{n-1}
+\frac{3n^2+17n+26}{3\cdot 2^{n+1}(n-2)!}q^{n-2}
+\cdots
\end{aligned}
\]
Then by inserting the values $A_0,A_1,A_2$ into~\eqref{eq:bound},
we get the following upper bound
\[
b_2(X)\le n+17+\frac{12}{n+1}.
\]
% Alternatively, we could also have used Remark~\eqref{rmk:RR_split} to obtain the
% expression.
When $n=2$ or $n=3$, it evaluates to 23 and is attained by $\mathrm K3^{[n]}$;
when $n=5$, it evaluates to 24 and is attained by $\OG_{10}$. In particular,
these are exactly the three known deformation types with this Riemann--Roch polynomial for which we have $c_2\in \Sym^2 H^2(X,\bR)$.
\end{example}

\begin{example}[$\Kum_n$-type]
We compute similarly the first three coefficients
\[
\begin{aligned}
\RR_{\Kum_n}(q) &= (n+1)\binom{q/2+n}n\\
&=\frac{n+1}{2^nn!}q^n
+\frac{(n+1)^2}{2^n(n-1)!}q^{n-1}
+\frac{(n+1)^2(3n+2)}{3\cdot 2^{n+1}(n-2)!}q^{n-2}
+\cdots
\end{aligned}
\]
and insert these three coefficients into~\eqref{eq:bound}. In this case, the
upper bound we get is
\[
b_2(X)\le n+5.
\]
When $n=2$, it is attained by $\Kum_2$; when $n=3$ it is attained by $\OG_6$.
Again, for these two types, we have $c_2\in \Sym^2 H^2(X,\bR)$.
%Note also that for $n=2$, the bound $b_2(X)\le 7$ is much stronger than the general bound $b_2(X)\le 23$ by Guan.
\end{example}

Another consequence of the inequality is the positivity of the generalized
Fujiki constants $C(c_2^2)$ and $C(c_4)$.
\begin{proposition}
Let $X$ be a hyperkähler manifold of dimension $2n$. The generalized Fujiki
constant $C(c_2^2)$ is always positive, and $C(c_4)$ is positive except
possibly when $n=2$ and $b_2(X)=3,4,5$ or
when $n=3$ and $b_2(X)=3$.
\end{proposition}
\begin{proof}
From the inequality~\eqref{eq:inequality_fujiki}, it is clear that $C(c_2^2)$
is positive.
% For $C(c_4)$ to be positive, it is equivalent to have
% \[
% \frac{nA_0A_2}{(n-1)A_1^2}>\frac37.
% \]
% % By~\eqref{eq:inequality_A}, we have
% \jieao{TODO}
% \[
% \frac{nA_0A_2}{(n-1)A_1^2}\ge\frac14\left(2-\frac2{b_2(X)+2n-2}\right).
% \]
% So we want the inequality
% \[
% \frac14\left(2-\frac2{b_2(X)+2n-2}\right)>\frac37
% \]
Assuming that we also have the strict inequality $b_2(X)+2n> 9$, which is satisfied except when $n=2$ and $b_2(X)\le 5$ or $n=3$ and $b_2(X)=3$,
then by~\eqref{eq:ineq_c2_2_and_c4_multiplied} we can easily check that $C(c_4)$ is positive as well.
\end{proof}
\begin{remark}
When $n=2$, these two generalized Fujiki constants are just Chern numbers, and
this is already known by the results of Guan \cite{Guan4dimHK}.
\end{remark}

\section{Orbifold examples}
\label{sec:orbifold_examples}
Theorem~\ref{thm:bound_b2_c2_Verbitsky} can also be generalized to the singular case, namely, to \emph{primitively symplectic orbifolds}.
\begin{definition}[{Fujiki \cite{FujikiOrbifold}, see also \cite[Definition~3.1]{FuMenet}}]
A compact Kähler orbifold $X$ of dimension $2n\ge 4$ with $\codim\Sing X\ge 4$ is called \emph{primitively symplectic} if
the smooth locus $X_{\mathrm{reg}} \coloneqq X \setminus \Sing X$ is endowed with a non-degenerated holomorphic $2$-form which is unique up to scalar.
\end{definition}
We also recall that a primitively symplectic orbifold $X$ is called \emph{irreducible} if the smooth locus $X_{\mathrm{reg}}$ is simply connected.

First, we check that the relevant notions for Theorem~\ref{thm:bound_b2_c2_Verbitsky} exist in the orbifold setting.
\begin{itemize}
    \item The second cohomology group $H^2(X,\bR)$ admits a canonical quadratic form of signature $(3,b_2(X)-3)$ \cite[Theorem~3.17]{MenetTorelli}.
    \item Generalized Fujiki constants still exist, as established by Menet in~\cite[Lemma~4.6]{MenetTorelli}.
    \item Consequently, we may define the Riemann--Roch polynomial using the generalized Fujiki constants of the Todd class
    \[
    \RR_X(q)\coloneqq\sum_{i=0}^n\frac{C(\td_{2n-2i})}{(2i)!}q^i=A_0q^n+\cdots+A_n.
    \]
\end{itemize}
\begin{theorem}
\label{thm:orbifold}
Theorem~\ref{thm:bound_b2_c2_Verbitsky} holds for a primitively symplectic orbifold $X$ of dimension $2n\ge4$.
\end{theorem}
\begin{proof}
The proof follows the exact same outline as in Section~\ref{sec:The_Inequality}, so we only indicate the key ingredients.
% We will follow the paper by Fu--Menet~\cite{FuMenet} and the notation therein.
\begin{itemize}
    \item An orbifold version of the Hodge--Riemann bilinear relation holds \cite[Proposition~2.14]{MenetTorelli}, which allows us to prove the analogue of Proposition~\ref{prop:fundamental_inequality}.
    \item Proposition~\ref{prop:fujiki_of_q} can be repeated word for word.
    \item An orbifold version of the Hitchin--Sawon formula exists due to Fu--Menet \cite[Proposition~4.2]{FuMenet}.
    For $n=2$, this provides precisely the orbifold version of Theorem~\ref{thm:RR_td_1/2} and Corollary~\ref{cor:relation_td1/2}.
More generally, the orbifold version of Theorem~\ref{thm:RR_td_1/2} also hold in higher dimensions, which can be obtained by following the same argument of Fu--Menet.
The key idea is that such relations are of a local nature and not affected by quotient singularities, and, therefore, generalize to the orbifold setting (see also the proof in Section~\ref{sec:proof_RR_td_1/2}).
    % In particular, when $n=2$, this gives the orbifold version of Corollary~\ref{cor:relation_td1/2}. One would expect that the more general result of Nieper-Wißkirchen should also hold for the singular case.
\end{itemize}
With these ingredients, one can conclude following the proof of Theorem~\ref{thm:bound_b2_c2_Verbitsky}.
\end{proof}
% We note that Theorem~\ref{thm:orbifold} should more generally hold for primitively symplectic orbifolds of arbitrary dimension $2n\ge 4$.
% Namely, the orbifold version of Theorem~\ref{thm:RR_td_1/2} should hold in higher dimensions, which can be obtained by following the same argument of Fu--Menet \cite[Proposition~4.2]{FuMenet}.
% The key idea is that such relations are of local nature and not affected by quotient singularities, and, therefore, generalize to the orbifold setting (see also the proof in Section~\ref{sec:proof_RR_td_1/2}).

We note that orbifold versions of the Gauss--Bonnet theorem and the Hirzebruch--Riemann--Roch theorem exist in dimension~4 (or more generally, for orbifolds with only isolated singularities), as proved by Blache in~\cite{Blache} (see \cite[Theorem~2.12 and Theorem~2.13]{FuMenet}): we have
\[
\chi_\top(X)=\int_{X} c_{2n}+\sum_{x\in \Sing(X)}\left(1-\frac1{|G_x|}\right),
\]
and for all $L\in\Pic(X)$,
\[
\chi(X,L) = \int_{X}\ch(L)\cdot \td(X) + \sum_{x\in\Sing(X)}\frac1{|G_x|}\sum_{g\in G_x\setminus
\set{e}} \frac1{\det(\Id-\rho_{x,T_{X}}(g))}.
\]
The Gauss--Bonnet theorem will be useful to deduce the generalized Fujiki constant $C(c_{2n})$ using the topological Euler characteristic, as they are no longer equal.
Similarly, we emphasize that the Riemann--Roch polynomial as defined above no longer gives the correct holomorphic Euler characteristic, due to the contribution from singular points: for example in dimension $4$ we have instead
\[
\forall L\in\Pic(X)\quad \chi(X,L) = \RR_X(q_X(L))+(3-C(\td_4)),
\]
where $C(\td_4)$ is no longer always equal to $3$.
\bigskip

Now we apply Theorem~\ref{thm:orbifold} to examine the examples studied in \cite[Section~5]{FuMenet}. We will use $a_m$ to denote the number of isolated cyclic quotient singularities of order~$m$.
% \begin{remark}
% The conceptual reason why Theorem~\ref{thm:bound_b2_c2_Verbitsky} remains valid also in the singular case is that this type of result holds pointwise and, therefore, generalizes to orbifolds. 
% \end{remark}

\begin{example}
Let $M'$ be the irreducible symplectic orbifold of dimension 4 with second Betti
number $b_2(M')=16$, also known as a \emph{Nikulin orbifold}~(\cite[Section~5.11]{FuMenet}; see also \cites{MenetBB, CGKK}). It has 28 isolated quotient singularities of order 2, i.e., $a_2=28$. The orbifold $M'$ has topological Euler
characteristic $\chi_\top(M')=212$ and Fujiki constant $C(1)=6$ \cite[Theorem~2.5, Proposition~2.40]{MenetBB}.

Using the orbifold Riemann--Roch and Gauss--Bonnet theorems, we get
\[
\begin{aligned}
\int_{M'} \td_4=\int_{M'}\frac{3c_2^2-c_4}{720}
&=\chi(M',\cO_{M'})- \sum_{x\in\Sing(M')}\frac1{|G_x|}\sum_{g\in G_x\setminus
\set{e}} \frac1{\det(\Id-\rho_{x,T_{M'}}(g))}\\
&=3-28\cdot \frac12\cdot\frac1{16}=\frac{17}8,
\end{aligned}
\]
and
\[
\int_{M'} c_4=\chi_\top(M')-\sum_{x\in \Sing(M')}\left(1-\frac1{|G_x|}\right)=
198.
\]
Therefore we may compute
\[
C(c_2^2) = 576, \quad C(c_4)=198.
\]
The orbifold Hitchin--Sawon formula gives the relation in
Corollary~\ref{cor:relation_td1/2}, from which we deduce that $C(c_2)=36$.
Hence we have obtained the Riemann--Roch polynomial of~$M'$: 
\[
\RR_{M'}(q) = \frac14q^2+\frac32q+\frac{17}8.
\]
Note that this polynomial was also computed directly from the geometry of $M'$ by Camere--Garbagnati--Kapustka--Kapustka~\cite[Theorem~1.3]{CGKK}.

Now if we insert the values into~\eqref{eq:bound}, we get
\[
b_2(X)\le 16,
\]
for any primitively symplectic orbifold $X$ with the same Riemann--Roch
polynomial as $M'$. The Nikulin orbifold $M'$ attains the upper bound, and we have
$c_2(M')\in\Sym^2 H^2(M',\bR)$. Note that the two roots of $\RR_{M'}(q)$ are
$-3\pm\frac{\sqrt2}2$, so they are not integers.
\end{example}

\begin{example}
Let $K'$ be the orbifold example in~\cite[Section~5.6]{FuMenet} with second Betti number $b_2(K')=8$ and
$a_2=36$: we have $\chi_\top(K')=108$ and $C(1)=8$. Similarly, we
compute
\[
C(c_2)=40,\quad
C(c_2^2)=480,\quad
C(c_4)=90,
\]
and
\[
\RR_{K'}(q)=\frac13q^2+\frac53q+\frac{15}8.
\]
Using~\eqref{eq:bound}, we get the bound
\[b_2(X)\le 8,\]
which again holds for any primitively symplectic orbifold with the same Riemann--Roch polynomial.
So the example $K'$ also attains the upper bound. The two roots are
$\frac{-10\pm\sqrt{10}}4$.

Note that surprisingly, the Beauville--Bogomolov-Fujiki form of $K'$ is odd and represents the value 1. If we take a line bundle $H$ with $q(c_1(H))=1$, after adding the correction term, the Riemann--Roch formula tells us that $\chi(K', H)=5$, so one could expect that the linear system $|H|$ gives a (rational) finite cover of $\bP^4$.
\end{example}

\begin{example}
The following examples are obtained as cyclic quotients of smooth
hyperkähler manifolds of $\mathrm K3^{[2]}$-type~\cite[Sections~5.2, 5.3, and 5.9]{FuMenet}, so they are primitively symplectic but not irreducible.
\begin{itemize}
\item Case $b_2(M^i_{11})=3$ for $i=1,2$ with $a_{11}=5$:
we have $\chi_\top(M^i_{11})=34$ and $C(1)=33$ for both $i=1,2$, so
\[
C(c_2)=30,\quad
C(c_2^2)=\frac{828}{11},\quad
C(c_4)=\frac{324}{11},
\]
and
\[
\RR_{M^i_{11}}(q)=\frac{11}8q^2+\frac54q+\frac3{11}
=\frac1{11}\RR_{\mathrm K3^{[2]}}(11q).
\]
\item Case $b_2(M_7)=5$ with $a_7=9$:
we have $\chi_\top(M_7)=54$ and $C(1)=21$, so
\[
C(c_2)=30,\quad
C(c_2^2)=\frac{828}7,\quad
C(c_4)=\frac{324}7,
\]
and
\[
\RR_{M_7}(q)=\frac78q^2+\frac54q+\frac37
=\frac17\RR_{\mathrm K3^{[2]}}(7q).
\]
\item Case $b_2(M_3)=11$ with $a_3=27$:
we have $\chi_\top(M_3)=126$ and $C(1)=9$, so
\[
C(c_2)=30,\quad
C(c_2^2)=276,\quad
C(c_4)=108,
\]
and
\[
\RR_{M_3}(q)=\frac38q^2+\frac54q+1
=\frac13\RR_{\mathrm K3^{[2]}}(3q).
\]
\end{itemize}
In all these cases, the bound we get is $b_2(X)\le 23$, which is not attained.
These are all equal to the bound for $\mathrm K3^{[2]}$, due to the fact that
the expression in~\eqref{eq:bound} is homogeneous in terms of the roots of
$\RR_X(q)$, hence will remain invariant after a change of variables.

In some sense, taking cyclic quotient does not produce genuinely ``new''
examples or Riemann--Roch polynomials.
\end{example}

\begin{example}
For the following examples, we could not find the values of the Fujiki constant
$C(1)$ in the literature. But a bound on $b_2$ can still be given, due to the equivalent expression in \eqref{eq:bound} which only involves $C(c_2^2)$ and $C(c_4)$. We will simply write the upper bound obtained as $b_2(X)\le B$, where $X$ is understood as a primitively symplectic orbifold with the same Riemann--Roch polynomial.
\begin{itemize}
\item Case $b_2(K_4')=6$ with $a_2=30,a_4=8$ and $\chi_\top(K_4')=66$ (the numerical invariants obtained in \cite[Section~5.4]{FuMenet} are not correct, see Appendix~\ref{app:K4} for a detailed analysis; see also \cite[Section~4.7]{Menet33}): we have
\[
\begin{gathered}
C(c_2)=10\sqrt{C(1)},\quad
C(c_2^2)=240,\quad
C(c_4)=45,\\
\RR_{K_4'}(q) = \frac{C(1)}{24}q^2+\frac{5\sqrt{C(1)}}{12}q + \frac{15}{16},
\quad
b_2(X)\le 8.
\end{gathered}
\]
Note also that $K_4'$ can be realized as a $\bZ/2$-quotient of the Kummer example $K'$, and indeed the two bounds that we obtained are the same.%
\footnote{One might expect $C(1)=16$ so we would have $\RR_{K_4'}(q)=\frac12\RR_{K'}(2q)$.}

%$a_2=45,a_4=2$ and $\chi_\top(K_4')=69$~\cite[Sec.~5.4]{FuMenet}: we have
%\[
%C(c_2)=\sqrt{142C(1)},\quad
%C(c_2^2)=330,\quad
%C(c_4)=45,
%\]
%and
%\[
%b_2(X)\le \frac{55}8=6.875.
%\]
%So $b_2(K_4')=6$ is the maximal possible but does not attain the bound.
%\jieao{I think I have found the error in their description: in particular this example $K'_4$ is a $\bZ/2$-quotient of $K'$ so the top Chern numbers should be those of $K'$ divided by 2, and the bound is still equal to 8 (like the other cyclic quotient examples). I think its Riemann--Roch polynomial should $\frac12\RR_{K'}(2q)$. I'm asking Grégoire about this.}

\item Case $b_2(K_3')=7$ with $a_3=12$ and $\chi_\top(K_3')=108$~\cite[Section~5.5]{FuMenet}: we have
\[
\begin{gathered}
C(c_2)=26\sqrt{C(1)/3},\quad
C(c_2^2)=540,\quad
C(c_4)=100,\\
\RR_{K_3'}(q) = \frac{C(1)}{24}q^2+\frac{13\sqrt{C(1)/3}}{12}q + \frac{19}{9},
\quad
b_2(X)\le \frac{135}{17}\approx 7.94
\end{gathered}
\]
So $b_2(K_3')=7$ is the maximal possible but does not attain the bound.

\item Case $b_2(Y_\KKK(D_3))=9$: the description of this example in~\cite{FuMenet} appears to be incorrect.%
\footnote{Namely, the orbifold is described as the quotient of an $S^{[2]}$ by some symplectic automorphisms forming the dihedral group $D_3$. But such a quotient would necessarily contain singularities in codimension 2.}

\begin{comment}
We think that instead of $a_3=35$, we should
have $a_2=44$ and $a_3=9$ (see the end for the argument). Then we have
$\chi_\top=86$ and
\[
C(c_2)=\frac23\sqrt{227C(1)},\quad
C(c_2^2)=\frac{748}3,\quad
C(c_4)=58,
\]
and
\[
b_2(X)\le \frac{187}{20}=9.35.
\]
So then $b_2(Y_\KKK(D_3))=9$ is the maximal possible but does not attain the
bound. (If we use $a_3=35$ instead, the condition~\eqref{eq:hypothesis} will
fail, which would be weird.)
\end{comment}

\item Case $b_2(Y_\KKK(\bZ/4\bZ))=10$ with $a_2=10, a_4=6$ and
$\chi_\top=140$~\cite[Table~1]{FujikiOrbifold}: we have
\[
\begin{gathered}
C(c_2)=8\sqrt{3C(1)},\quad
C(c_2^2)=486,\quad
C(c_4)=\frac{261}2,\\
\RR_{Y_\KKK(\bZ/4\bZ)}(q) = \frac{C(1)}{24}q^2+\frac{\sqrt{3C(1)}}{3}q + \frac{59}{32},
\quad
b_2(X)\le \frac{54}5=10.8
\end{gathered}
\]
So $b_2(Y_\KKK(\bZ/4\bZ))=10$ is the maximal possible but does not attain the bound.

\item Case $b_2\big(Y_\KKK\big((\bZ/2\bZ)^2\big)\big)=14$ with $a_2=36$ and
$\chi_\top=180$~\cite[Table~1]{FujikiOrbifold}: we have
\[
\begin{gathered}
C(c_2)=8\sqrt{3C(1)},\quad
C(c_2^2)=504,\quad
C(c_4)=162,\\
\RR_{Y_\KKK((\bZ/2\bZ)^2)}(q) = \frac{C(1)}{24}q^2+\frac{\sqrt{3C(1)}}{3}q + \frac{15}{8},
\quad
b_2(X)\le 14.
\end{gathered}
\]
So the bound is attained in this example.
\end{itemize}
\end{example}

\begin{example}[Kim]
This example was studied by Kim in~\cite{KimDualLagrangian}: let $X$ be a hyperkähler fourfold of $\Kum_2$-type admitting a Lagrangian fibration. We consider its dual Lagrangian fibration $\check X$. It is a primitively symplectic orbifold with only isolated quotient singularities, with $a_3=36$ and $\chi_\top(X)=108$.

% However, the analysis in {\it loc.\ cit.} of the singularities of $\check X$ contains an error: the group action admits 108 fixed points on $X$, and every other 3 of them are identified after the quotient. So one should have $a_3=36$, that is, $\check X$ admits 36 isolated cyclic quotient singularities of order~3, instead of just 18 of them as claimed in {\it loc.\ cit}. Since $\chi_\top(X)=108$, we may conclude that $\chi(\check X)=108/3=36$, which is consistent with the description of the cohomology.

We compute the numerical invariants. By the orbifold Gauss--Bonnet theorem, we have
$C(c_4)=\chi_{\mathrm{top}}-a_3\cdot \frac23=12$. Then by the orbifold Riemann--Roch theorem, we have
$\frac1{720}\big(3C(c_2^2)-C(c_4)\big)=3-a_3\cdot\frac13\cdot \frac29=\frac13$, hence
$C(c_2^2)=84$. This already gives us the bound on the second Betti number
\[
b_2(\check{X})\le \frac{10}{\frac{84}{12}-2}-2\cdot 2+9=7,
\]
which is attained by the dual Lagrangian fibration $\check X$.

Kim showed that the \emph{small} Fujiki constant $c_{\check X}$ of the dual Lagrangian fibration $\check X$ is $1/c_X$, so $C(1_{\check X})=\frac13\cdot 3=1$ in the dual $\Kum_2$ case.
Then by the orbifold Hitchin--Sawon formula, we may compute that $C(c_2)=6$. Hence the Riemann--Roch polynomial is given by
\[
\RR_{\check{\Kum}_2}(q) = \frac1{24}q^2+\frac1{4}q+\frac13=\frac19\RR_{\Kum_2}(q).
\]
In particular, for a line bundle $H$ with square $q(c_1(H))=6$, we can use the Riemann--Roch formula with the correction term to compute
$\chi(\check X, H)=6$. So one could expect that the linear system $|H|$ gives a hypersurface (or a cover thereof) in $\bP^5$.

\end{example}

\begin{comment}
\begin{remark}[The case $b_2=9$]
The three equation (33), (34), (35) are linearly dependent, so we can only
solve that
\[
N_1+2N_2=14.
\]
We now try to determine the values of $N_1$ and $N_2$.

We take a symplectic involution $f_2$ and an automorphism $f_3$ of order 3 that
generate the action of $H$. The other non-trivial elements are $f_2f_3$ and
$f_2f_3^2$, which are involutions, and $f_3^2$, which is of order 3.
A symplectic involution on a K3 surface admits exactly 8 fixed points, and a
symplectic automorphism of order 3 admits exactly 6 fixed points.

The three sets of 8 fixed points for the three involutions must be disjoint:
for example, since we have $f_2\cdot(f_2f_3)=f_3$, a common fixed point of
$f_2$ and $f_2f_3$ will also be fixed by $f_3$, which is impossible. Hence the
restriction of the action of $H$ to these 24 fixed points is generated by
$f_3$, and therefore we have $N_1\ge 8$.

Similarly, the restriction of the action of $H$ to the 6 fixed points of $f_3$
is generated by $f_2$, so we have $N_2\ge 3$. We may then conclude that $N_1=8,
N_2=3$.

Then finally, we have $a_2=\binom{N_1}2+2N_1=44$ and
$a_3=\binom{N_2}2+2N_2=9$.
\end{remark}
\end{comment}

%\section{Notes, computations, ideas, remarks, etc.}

\section{Generalized Fujiki constants for known smooth examples}
\label{sec:generl_fujiki_known_smooth}
In this section, we give an account for the generalized Fujiki constants
$C(c_\lambda)$ of characteristic classes $c_\lambda\coloneqq
c_2^{\lambda_2}c_4^{\lambda_4} \cdots c_{2n}^{\lambda_{2n}}$ for all known
deformation types of hyperkähler manifolds.
\subsection{$\mathrm{K3^{[n]}}$ and $\Kum_n$} The results are classical for the
two infinite families. In the $\mathrm K3^{[n]}$-case,
the method in Ellingsrud--Göttsche--Lehn~\cite{EGL} can be used to
compute all the generalized Fujiki constants using a computer for small $n$. 
A similar algorithmic method can be used to
treat the $\Kum_n$-case, with some slight modifications based on the work of
Nieper-Wißkirchen~\cite[Section~4.2.3]{NieperWißkirchenbook}.
An implementation for these algorithms in {\tt Sage} can be found in \cite{bott}.
Closed formulae for the values $C(c_{2k})$ for both families were recently established in~\cite[Theorem~4.2, Proposition~4.3]{CaoOberdieckToda}.
\subsection{$\OG_6$}
By Corollary~\ref{cor:first_four_C}, the generalized Fujiki constants for
characteristic classes of degree~$\le 4$ for $\OG_6$ are the same as those for
$\Kum_3$, since they share the same Riemann--Roch polynomial. Since the Chern
numbers of $\OG_6$ are also known~\cite[Proposition~6.8]{MRSOG6}, we can obtain all of them:
\[
\renewcommand{\arraystretch}{1.2}
\begin{array}{|c|c|c|cc|ccc|}
\hline
\alpha&1&c_2&c_4&c_2^2&c_6&c_4 c_2&c_2^3\\
\hline
C(\alpha)&60&288&480&1920&1920&7680&30720\\
\hline
\end{array}
\renewcommand{\arraystretch}{1}
\]
Alternatively, since for $\OG_6$-type the second Chern class $c_2$ lies in the
Verbitsky component (namely, $c_2(\OG_6)=2\q$ by comparing the generalized Fujiki constant using Proposition~\ref{prop:fujiki_of_q}), Corollary~\ref{cor:td_2n-2_verbitskycomponent} shows that the class $\td^{1/2}_4$ also lies in $\SH(X,\bR)$. Now $\td^{1/2}_4$ is a linear combination of $c_2^2$ and $c_4$, so the same may be said for the class $c_4$. Then we can use
Proposition~\ref{prop:fujiki_of_q} to determine that $c_4(\OG_6)=\q^2$, which then allows us to also compute $C(c_4 c_2)$ and $C(c_2^3)$. Finally we can use $C(\td_6)=4$ to solve the Euler characteristic $C(c_6)$.
\begin{proposition}
For hyperkähler manifolds of $\OG_6$-type, all Chern classes
$c_2,c_4,c_6$ lie in the Verbitsky component. We have
\[
c_2(\OG_6)=2\q,\quad
c_4(\OG_6)=\q^2,\quad
c_6(\OG_6)=\tfrac12\q^3.
\]
\end{proposition}

\subsection{$\OG_{10}$}
The question for $\OG_{10}$ might seem difficult at first, as there are many
more unknown Fujiki constants to determine. It turns out to be quite easy, due
to the following observation.
\begin{proposition}
For hyperkähler manifolds of $\OG_{10}$-type, all Chern classes
$c_2,\dots,c_{10}$ lie in the Verbitsky component. We have
\[
\begin{gathered}
c_2(\OG_{10}) =\tfrac32\q,\quad
c_4(\OG_{10}) =\tfrac{15}{16}\q^2,\quad
c_6(\OG_{10}) =\tfrac{21}{64}\q^3,\\
c_8(\OG_{10}) =\tfrac{237}{3328}\q^4,\quad
c_{10}(\OG_{10}) =\tfrac{27}{2560}\q^5.
\end{gathered}
\]
\end{proposition}
\begin{proof}
We use the LLV decomposition of the cohomology obtained in
\cite[Theorem~3.26]{GKLRLLV}
\[
H^*(\OG_{10},\bQ)=V_{(5)}\oplus V_{(2,2)} \quad\text{as $\so(4,22)$-modules}.
\]
We are interested in the second component, which only contributes to cohomological degree $k$ for $k\in\set{6,8,10,12,14}$.

For a generic $X$ in the moduli space, the (special) Mumford--Tate algebra is
the maximal possible and is isomorphic to $\so(3,21)$. Using the branching
rules, we get the following decompositions of $\so(3,21)$-modules/Hodge
structures ($H^{12}$ and $H^{14}$ are omitted by symmetry)
\[
\begin{aligned}
H^{6}(X,\bQ) &= \SH^{6}(X,\bQ)\oplus V_{(2)},\\
H^{8}(X,\bQ) &= \SH^{8}(X,\bQ)\oplus V_{(2,1)}\oplus V_{(1)},\\
H^{10}(X,\bQ)&= \SH^{10}(X,\bQ)\oplus V_{(2,2)}\oplus V_{(2)}\oplus
V_{(1,1)}\oplus \bQ.
\end{aligned}
\]
In other words, up to multiplying by a non-zero scalar, there is only one Hodge class $\eta\in H^{10}(X,\bQ)$ that lies in $\SH(X,\bQ)^\perp$ for a generic $X$. In particular, this
means that all the Chern classes $c_2,\dots,c_{10}$ lie in the Verbitsky
component.

For a generic $X$, the only Hodge classes in the Verbitsky components are
multiples of powers of~$\q$, so each Chern class $c_{2k}$ is a
multiple of $\q^k$, and it suffices to determine the scalar.
Since the values of $C(\q^k)$ are known by Proposition~\ref{prop:fujiki_of_q}, it suffices to determine the generalized Fujiki constant $C(c_{2k})$.
We explain how to do this starting from smaller~$k$: we use Corollary~\ref{cor:first_four_C} to determine $C(c_2)$ and $C(c_4)$.
Once all $c_{2i}$ for $i<k$ are known, we study the class
$\td^{1/2}_{2k}$, whose generalized Fujiki constant $C(\td^{1/2}_{2k})$ is
known by Theorem~\ref{thm:RR_td_1/2} and whose only unknown term
is a given multiple of $c_{2k}$. Therefore we will be able to uniquely
determine $C(c_{2k})$ and thus $c_{2k}$ itself.
\end{proof}

It is then straightforward to compute the generalized Fujiki constants, which
we include for the reader's convenience.
\[
\resizebox{\hsize}{!}{
$
\renewcommand{\arraystretch}{1.2}
\begin{gathered}
\begin{array}{|c|c|c|cc|ccc|ccccc|}
\hline
\alpha&1&c_2&c_4&c_2^2&c_6&c_4 c_2&c_2^3&c_8&c_6c_2&c_4^2&c_4c_2^2&c_2^4\\
\hline
C(\alpha)&945&5040&13500&32400&26460&113400&272160&49770&343980&614250&
1474200&3538080\\
\hline
\end{array}\\
\begin{array}{|ccccccc|}
\hline
c_{10}&c_8c_2&c_6c_4&c_6c_2^2&c_4^2c_2&c_4c_2^3&c_2^5\\
\hline
176904&1791720&5159700&12383280&22113000&53071200&127370880\\
\hline
\end{array}
\end{gathered}
\renewcommand{\arraystretch}{1}
$
}
\]
Note that the Chern numbers for $\OG_{10}$ have already been computed by
Cao--Jiang in the appendix of~\cite{OrtizRRPolynomials}.

It is remarkable that the knowledge of the Riemann--Roch polynomial together with the
assumption that all Chern classes lie in the Verbitsky component
allow us to completely determine the second Betti number as well as all the generalized Fujiki
constants, in particular all the Chern numbers including the Euler
characteristic $C(c_{2n})=\int_X c_{2n}$.

\section{Further discussions}
\label{sec:further_discussions}

\begin{comment}
\subsection{The inequalities in Jiang's paper}
\begin{remark}
Jiang showed in his paper that the class
\[
\mathrm{tp}_2\coloneqq \frac1{24}\left( c_2 -
\frac{2(2n-1)C(c_2)}{nC(1)q(\sigma+\sigmabar)}\sigma\sigmabar\right)
\]
is $(\sigma+\sigmabar)$-primitive. This can also be proved by writing out an
orthonormal basis for $H^2(X,\bR)$: we let $e_1$ be the class of the Kähler
form, and $e_2+\sqrt{-1}e_3$ be the class of $\sigma$. Then we have
$\sigma+\sigmabar=2e_2$ so $q(\sigma+\sigmabar)=4$, and
$\sigma\sigmabar=e_2^2+e_3^2$. To verify that $\mathrm{tp}_2$ is
$e_2$-primitive, we just need to show that $\mathrm{tp}_2\cdot e_2^{2n-3}\cdot
e_i$ vanishes for all $i$.
Recall that the projection of $c_2$ to the
Verbitsky component is given by
\[
\frac{C(c_2)}{C(\q)}\q = \frac{(2n-1)C(c_2)}{(b+2n-2)C(1)}\cdot
(e_1^2+e_2^2+e_3^2-e_4^2-\cdots-e_b^2).
\]
Using the Fujiki relations, the only non-trivial
case is when $i=2$, in which case we may proceed as in the proof of Proposition~\ref{prop:fujiki_of_q} and compute
that $\mathrm{tp}_2\cdot e_2^{2n-2}=0$.

On the other hand, if we use Hodge--Riemann relation on $\mathrm{tp}_2$, the
obtained inequality would not be tight: the projection of $c_2$ has rank
$b_2(X)$ while $\sigma\sigmabar=e_2^2+e_3^2$ has only rank 2.
Indeed we would get the following inequality
\[
C(c_2^2)\ge\frac{(n-1)(2n-1)C(c_2)^2}{n(2n-3)C(1)},
\]
which is strictly weaker than \eqref{eq:inequality_fujiki}.
\end{remark}
\end{comment}

We see that the Riemann--Roch polynomial $\RR_X(q)$ of a hyperkähler manifold $X$ is a very important notion: it puts strong topological restriction on $X$, namely an upper bound for the second Betti number. We now formulate some conjectures on the shape of such polynomials and discuss some possible ways of studying them.

Recall from Theorem~\ref{thm:RR_td_1/2} that the polynomial $\RR_{X,1/2}(q)$ factors as a $n$-th power. The proof by Nieper-Wißkirchen \cite{NieperWissHRRonIHS} uses the machinery of Rozansky--Witten invariants.
We will briefly explain the proof, and discuss the possibility of using this method to study the Riemann--Roch polynomial $\RR_{X}(q)$.
\subsection{Conjectural form of the Riemann--Roch polynomial}
Motivated by the above discussions, we speculate about the general shape of the Riemann--Roch polynomial of certain symplectic varieties.

We make the following conjecture. Similar conjectures have already been formulated by Ríos Ortiz and Jiang in~\cite[Conjecture~1.3]{JiangRR}.
\begin{conjecture}
\label{conj:general_form_RR_polynomial}
Let $X$ be a primitively symplectic orbifold of dimension $2n$.
\begin{enumerate}
\item
The Riemann--Roch polynomial $\RR_X(q)$ has $n$ distinct negative real roots forming an arithmetic sequence. 
\item
If $X$ is smooth, then its Riemann--Roch polynomial $\RR_X(q)$ has even negative integer roots $\lambda_1, \dots, \lambda_n$ satisfying $\lambda_i-\lambda_{i-1}=2$. 
\end{enumerate}
\end{conjecture}
The second point is a slight strengthening of \cite[Conjecture~1.3(3)]{JiangRR}.
We do not have a particular good reason other than the naïve observation that it holds for the two known examples and the fact that the quadratic form is even in the known cases.
Note that it fails already in the case of four-dimensional orbifolds as demonstrated in Section~\ref{sec:orbifold_examples} and should necessarily involve the smoothness assumption.
Conceptually, the contribution from singular points will break the nice arithmetic property a smooth hyperkähler manifold might enjoy.

By Theorem~\ref{thm:bound_b2_c2_Verbitsky}~(1), Conjecture~\ref{conj:general_form_RR_polynomial}~(1) would imply the inequality \eqref{eq:hypothesis} and therefore yield the bound on the second Betti number.

\begin{comment}
{\color{brown}
Conversely, since for $n=2$ we already have a bound on the second Betti number, which is obtained by Guan for the smooth case, and by Fu--Menet~\cite{FuMenet} for the general case, Conjecture~\ref{conj:general_form_RR_polynomial}~(1) is verified when $n=2$. The key ingredient is the positivity of $C(\ch_4)$ (see also~\cite[Lem.\ 4.6]{OberdieckSongVoisin} and \cite[Thm.\ 7]{Sawon2021}).
\begin{lemma}
Conjecture~\ref{conj:general_form_RR_polynomial}~(1) holds for $n=2$.
\end{lemma}
\begin{proof}
When $n=2$, the Riemann--Roch polynomial has the following form
\[
\RR_X(q) = \frac{C(1)}{24}q^2+\frac{C(c_2)}{24}q+\frac{3C(c_2^2)-C(c_4)}{720} = \frac{C(1)}{24}\left(q+\frac{C(c_2)}{2C(1)}\right)^2-\frac1{120}C(\ch_4),
\]
where we used the relation from Corollary~\ref{cor:relation_td1/2}. So it suffices to show the positivity of $C(\ch_4)$.

When $X$ is smooth, we have $\frac1{240}C(c_2^2)-\frac1{720}C(c_4)=C(\td_4)=3$, hence
\[
C(\ch_4)=\frac1{12}C(c_2^2)-\frac16C(c_4)=-\frac5{36}C(c_4)+60,
\]
so we want $C(c_4)\le 432$. By Salamon's relation and the bound by Guan, we have 
\[
C(c_4)=\chi_\top(X)=b_4-2b_3+2b_2+2\le b_4 +b_3+2b_2+2= 12b_2+48\le 324.
\]

For the orbifold case, we need to keep track of the contributions from singular points. We have $\frac1{240}C(c_2^2)-\frac1{720}C(c_4)=C(\td_4)=3-\sum_{x\in\Sing(X)}$
{\color{red} For orbifolds it seems to be more complicated, so maybe we should just leave this lemma out. I made the remark also in Sec. 5.4, the $n=2$ case.}
\end{proof}
}
\end{comment}

\subsection{Rozansky--Witten invariants}
We give a very rough overview of parts of Rozansky--Witten theory that we want to employ. For proofs, details and a general overview we refer mainly to the book \cite{NieperWißkirchenbook}. See also~\cites{JiangRR, SawonThesis, HitchinSawon}. 

After choosing a symplectic form $\sigma \in H^0(X,\Omega_X^2)$, the \textit{Rozansky--Witten weight system} $\RW_\sigma$ is a ring homomorphism
\begin{equation}
    \RW_\sigma \colon B \to H^\ast(X,\bC),
\end{equation}
where $B$ denotes the graph homology space, i.e., the $\bC$-algebra spanned by all unitrivalent graphs modulo the antisymmetry and IHX relation. Important graphs are $\ell$, the unique univalent graph with two vertices, $\Theta$, the trivalent graph $\ominus$ with two vertices, and the $2k$-wheels $w_{2k}$ which, for example, looks like 
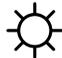
\begin{equation*}\resizebox{.05\hsize}{!}{$\sun$}
\end{equation*}
for $k=4$.

Using the $2k$-wheels, we can define the wheeling element
\[
\Omega \coloneqq \exp \left( \sum_{k=1}^\infty b_{2k}w_{2k} \right)
\]
contained in the completion $\hat{B}$ of $B$ with $b_{2k}$ the modified Bernoulli numbers. We have
\begin{itemize}
    \item $\RW_\sigma(\ell)=2\sigma$,
    \item $\RW_\sigma(\Theta)=b_{\Theta} \left[\frac{2\sigmabar}{q(\sigma+\sigmabar)}\right]$, where $b_\Theta=48\,r_X=\frac{2(2n-1)C(c_2)}{C(1)}$~\cite[Proposition~7]{NieperWissHRRonIHS}, %\frac{48}{\lambda_\sigma} \sigmabar$ for $\lambda_\sigma$ as in \cite[Def.\ 17]{NieperWissHRRonIHS},
    \item $\RW_\sigma(w_{2k})=-(2k)!\,\ch_{2k}$,
    \item $\RW_\sigma(\Omega)=\td^{1/2}.$
\end{itemize}

%The subalgebra $\hat{B'}\subset \hat{B}$ consisting of graphs not having $\ell$ as a component is equipped with a differential operator
%\[
%\partial \colon B' \to B
%\]
%sending $\Gamma$ to the sum of all graphs obtained by gluing two univalent vertices of $\Gamma$.

There is a bilinear product $\inner{-,-}$ on the graph homology space defined by summing over all possible ways of gluing all univalent vertices of the graphs under consideration, see \cite[Definition~2.39]{NieperWißkirchenbook} for a precise account. 
% The following result is usually called the wheeling theorem.
One form of the Wheeling Theorem is the following~\cite[Corollary~2.3]{NieperWißkirchenbook}.
\begin{theorem}
%The wheeling element $\Omega$ satisfies
%\[
%\partial \Omega = \frac{\Theta}{48} \Omega \in \hat{B}.
%\]
The map
\[
\inner{\Omega,-}\colon x\mapsto \inner{\Omega,x}
\]
respects the ring structure on $B$ given by disjoint union.
\end{theorem}

There is also a bilinear product $\inner{-,-}_\sigma$ defined on the cohomology \cite[Definition~3.9]{NieperWißkirchenbook}, which depends on the symplectic form $\sigma$ chosen. We use the subscript $\sigma$ to emphasize this dependence. The map $\RW_\sigma$ respects the two bilinear products~\cite[Proposition~3.4]{NieperWißkirchenbook}
\[
\RW_\sigma(\inner{x,y})=\inner{\RW_\sigma(x),\RW_\sigma(y)}_\sigma.
\]
This is the crucial result which allows us to transport relations present inside the graph homology space to the cohomology of $X$.

% \subsection{Generalized Fujiki constants via Rozansky--Witten theory}

\bigskip

Generalized Fujiki constants naturally appear in the study of Rozansky--Witten invariants, which can already be seen in the above formula for $\RW_\sigma(\Theta)$.
The key idea for the formula is that $\RW_\sigma(\Theta)$ is a class in $H^{0,2}(X)$,
which is generated by $[\sigmabar]$. So we can uniquely determine the class just by a scalar. To determine this number, one could cup the two classes with
$\exp(\sigma+\sigmabar)$ and compare the integral.

To illustrate this method, we determine the value of $\RW_\sigma(\Theta_2)$, where $\Theta_2$ is the necklace graph with two beads.
\begin{proposition}
\label{prop:RW_Theta2}
We have
\[
\begin{aligned}
\RW_\sigma(\Theta_2)&=-\frac{4\int (c_2^2-2c_4)\exp(\sigma+\sigmabar)}
{5n(n-1)\int \exp(\sigma +\sigmabar)}[\sigmabar]^2\\
&=-\frac{4(2n-1)(2n-3)C(c_2^2-2c_4)}{5C(1)}\left[\frac{2\sigmabar}{q(\sigma+\sigmabar)}\right]^2.
\end{aligned}
\]
\end{proposition}
\begin{proof}
Using the definition of the pairing $\inner{-,-}$ on the graph homology space, one can verify that
\[
\inner{w_4,\ell^2} = 20\Theta_2.
\]
Hence
\[
\begin{aligned}
\RW_\sigma(\Theta_2) &= \RW_\sigma\left(\frac1{20}\inner{w_4, \ell^2}\right)=
\frac1{20}\inner{\RW_\sigma(w_4), \RW_\sigma(\ell^2)}_\sigma\\
&=\frac1{20}\inner{-24\left(\tfrac1{12}c_2^2-\tfrac16c_4\right), 4\sigma^2}_\sigma=-\frac25\inner{c_2^2-2c_4,\sigma^2}_\sigma=-\frac45\inner{c_2^2-2c_4,\exp\sigma}_\sigma.
\end{aligned}
\]
Cupping it with $\exp(\sigma+\sigmabar)$ and comparing the integral, we get
\[
\RW_\sigma(\Theta_2) = -\frac{4\int\inner{c_2^2-2c_4,\exp\sigma}_\sigma\exp(\sigma +
\sigmabar)}{5\int\sigmabar^2\exp(\sigma+\sigmabar)}[\sigmabar]^2.
\]
For the denominator, we can simplify it as
\[
\begin{aligned}
\int_X\sigmabar^2\exp(\sigma+\sigmabar)&=
\int_X\sigmabar^2\frac1{(2n-2)!}(\sigma+\sigmabar)^{2n-2}\\
% &=\int_X \sigmabar^2\frac1{(2n-2)!}\binom{2n-2}{n}\sigma^n\sigmabar^{n-2}\\
&=\int_X \frac1{n!(n-2)!}(\sigma\sigmabar)^{n}\\
&=n(n-1)\int_X \exp(\sigma+\sigmabar).
\end{aligned}
\]
For the numerator, we use the following equality~\cite[Lemma~3.4]{NieperWißkirchenbook}
\[
\int_X\inner{\alpha,\exp\sigma}_\sigma\exp(\sigma+\sigmabar) = \int_X
\alpha\exp(\sigma+\sigmabar).
\]
This shows the first equality that we want to prove.

For the second equality, we note that for a class of type $(2j,2j)$, the Fujiki
relations give
\[
\int_X \alpha \exp(\sigma+\sigmabar)
= \int_X \alpha \cdot \frac1{(2n-2j)!}(\sigma+\sigmabar)^{2n-2j}
= \frac{C(\alpha)}{(2n-2j)!}q(\sigma+\sigmabar)^{n-j}.
\]
Taking $\alpha$ to be $1_X$ and $c_2^2-2c_4$ respectively, we get the desired
equality.
\end{proof}
In general, for a trivalent graph $\Gamma$ with $2k$ vertices, there is a number $b_\Gamma$ independent of the symplectic form $\sigma$ chosen, such that we have
\[
\RW_\sigma(\Gamma)=b_\Gamma\left[\frac{2\sigmabar}{q(\sigma+\sigmabar)}\right]^k\in H^{0,2k}(X).
\]
For example, we have obtained that
\[
b_\Theta=\frac{2(2n-1)C(c_2)}{C(1)}, \quad
b_{\Theta_2}=-\frac{4(2n-1)(2n-3)C(c_2^2-2c_4)}{5C(1)}.
\]
This is the same notation used by Sawon in~\cites{SawonThesis,Sawon2021}, although he only used the letter $b_\Gamma$ for graphs with exactly $2n$ vertices and referred to those as the \emph{Rozansky--Witten invariants} of $X$. By the properties of the map $\RW_\sigma$, the values $b_\Gamma$ are multiplicative with respect to disjoint union.

\bigskip

There is another way to obtain the value of $\RW_\sigma(\Theta_2)$. Namely
\[
\RW_\sigma(2\Theta_2) = \RW_\sigma(\inner{w_2,w_2})=4 \inner{c_2,c_2}_\sigma
\]
where we used the relation $\RW_\sigma(w_{2})=2c_2$. We therefore obtain from Proposition~\ref{prop:RW_Theta2} the equality
\[
\inner{c_2, c_2}_\sigma=\frac{b_{\Theta_2}}{2}\left[\frac{2\sigmabar}{q(\sigma+\sigmabar)}\right]^2 \in {H}^4(X,\mathcal{O}_X).
\]
We expect that this equality is equivalent to the equality obtained in Corollary~\ref{cor:relation_td1/2}, but have not pursued this further.

\subsection{Proof of Theorem~\ref{thm:RR_td_1/2}}
\label{sec:proof_RR_td_1/2}
Using the map $\RW_{\sigma}$ and the Wheeling Theorem, we can obtain a very conceptual proof and see why the polynomial $\RR_{X,1/2}(q)$ factorizes as an $n$-th power.
\begin{proof}[Proof of Theorem~\ref{thm:RR_td_1/2}]
For a class $\alpha$ of degree $(2k,2k)$ admitting a generalized Fujiki
constant, we follow the same method as in the proof of Proposition~\ref{prop:RW_Theta2} to compute
\begin{equation}
\label{eq:alpha_bilinear}
\begin{aligned}
\inner{\alpha,(2\sigma)^k}_\sigma&=
2^kk!\inner{\alpha,\exp\sigma}_\sigma\\
&=2^kk!\frac{\int\inner{\alpha,\exp\sigma}_\sigma\exp(\sigma+\sigmabar)}
{n(n-1)\cdots(n-(k-1))\int\exp(\sigma+\sigmabar)}[\sigmabar]^k\\
&=\frac{2^k}{\binom nk}\frac{\int\alpha\exp(\sigma+\sigmabar)}
{\int\exp(\sigma+\sigmabar)}[\sigmabar]^k\\
&=\frac{2^k}{\binom nk}\frac{\frac{C(\alpha)}{(2n-2k)!}
q(\sigma+\sigmabar)^{n-k}}{\frac{C(1)}{(2n)!}
q(\sigma+\sigmabar)^n}[\sigmabar]^k\\
&=\frac{1} {\binom
nk\frac{C(1)}{(2n)!}}\frac{C(\alpha)}{(2n-2k)!}
\left[\frac{2\sigmabar}{q(\sigma+\sigmabar)}\right]^k.
\end{aligned}
\end{equation}
We can take $\alpha$ to be $\td^{1/2}_{2k}$, which gives us
\[
\begin{aligned}
\frac{C(1)}{(2n)!}\inner{\td^{1/2},(1+2\sigma)^n}_\sigma&=
\frac{C(1)}{(2n)!}\inner{\td^{1/2},\sum_{k=0}^n\binom{n}{k}(2\sigma)^k}_\sigma\\
&= \sum_{k=0}^n\binom{n}{k}\frac{C(1)}{(2n)!}\inner{\td^{1/2}_{2k},(2\sigma)^k}_\sigma\\
&= \sum_{k=0}^n \frac{C(\td^{1/2}_{2k})}{(2n-2k)!}
\left[\frac{2\sigmabar}{q(\sigma+\sigmabar)}\right]^k\\
&=\RR'_{X,1/2}\left(\left[\frac{2\sigmabar}{q(\sigma+\sigmabar)}\right]\right).
\end{aligned}
\]
Here $\RR'_{X,1/2}(q) \coloneqq q^n\RR_{X,1/2}(1/q)$ is the polynomial obtained by reversing the
coefficients. The polynomial evaluated at the class $\left[\frac{2\sigmabar}{q(\sigma+\sigmabar)}\right]$ is an element in the cohomology ring, with terms in various degrees.

On the graph homology side, the Wheeling Theorem provides the relation 
\[
\inner{\Omega,\left(1+\ell\right)^n}=
\inner{\Omega,1+\ell}^n.
\]
Since the Rozansky--Witten invariant $\RW_{\sigma}$ is a ring homomorphism respecting the bilinear form $\inner{-,-}$, we get
\[
\inner{\td^{1/2}, (1+2\sigma)^n}_\sigma=\inner{\td^{1/2},1+2\sigma}_\sigma^n.
\]
Hence the polynomial $\RR_{X,1/2}(q)$ must indeed factorize as an $n$-th power.
\end{proof}

\subsection{Riemann--Roch polynomial via RW invariants}
Following the idea of the proof of Theorem~\ref{thm:RR_td_1/2}, if we want to study the Riemann--Roch polynomial $\RR_X$,
we should replace $\alpha$ with $\td_{2k}$ in~\eqref{eq:alpha_bilinear}: summing over all $k$ we get similarly
\[
\frac{C(1)}{(2n)!}\inner{\td,(1+2\sigma)^n}_\sigma=\RR'_X\left(\left[\frac{2\sigmabar}{q(\sigma+\sigmabar)}\right]\right).
\]
So for the same strategy to work, we need to study how the graph homology element
\[
\inner{\Omega^2,\left(1+\ell\right)^n}
\]
might potentially factorize into linear terms. Since the multiplication for the graph homology classes is the disjoint union, this would unfortunately not be possible in general. Below we compute its value for $n\le 4$:
%\[
%\begin{aligned}
%\inner{\Omega^2,\ell} &= \frac1{12}\Theta\\
%\inner{\Omega^2,\ell^2} &= \frac1{12^2}(\Theta^2+\Theta_2)\\
%\inner{\Omega^2,\ell^3} &= \frac1{12^3}(\Theta^3+3\Theta\Theta_2)\\
%\inner{\Omega^2,\ell^4} &= \frac1{12^4}(\Theta^4+6\Theta^2\Theta_2+3\Theta_2^2
%+\tfrac{144}{25}\Xi-\tfrac{162}{25}\Theta_4)
%\end{aligned}
%\]
\[
\begin{aligned}
\inner{\Omega^2,1+\ell}&=1+\frac 1{12}\Theta,\\
\inner{\Omega^2,(1+\ell)^2}&=1+\frac 1{12}2\Theta+\frac
1{12^2}(\Theta^2+\Theta_2),\\
\inner{\Omega^2,(1+\ell)^3}&=1+\frac 1{12}3\Theta+\frac
1{12^2}3(\Theta^2+\Theta_2)+\frac1{12^3}(\Theta^3+3\Theta\Theta_2),\\
\inner{\Omega^2,(1+\ell)^4}&=1+\frac 1{12}4\Theta+\frac
1{12^2}6(\Theta^2+\Theta_2)+\frac1{12^3}4(\Theta^3+3\Theta\Theta_2)\\
&\quad +\frac1{12^4}(\Theta^4+6\Theta^2\Theta_2+3\Theta_2^2
+\tfrac{144}{25}\Xi-\tfrac{162}{25}\Theta_4),
\end{aligned}
\]
where $\Xi$ is the extra graph for $n=4$. We study the implications on the Riemann--Roch polynomial.

%(This is the same as Nieper-Wißkirchen's $b_\Gamma$ or Sawon's $\beta_\Gamma$ up to some constant involving $8\pi$. The choice of notation is really not easy...)

\begin{itemize}
\item When $n=2$, we get
\[
\RR_X(q) = \frac{C(1)}{(2\cdot 2)!}\left(q^2+\frac1{12}2b_{\Theta}q+\frac1{12^2}(b_{\Theta}^2+b_{\Theta_2})\right).
\]
For the polynomial to admit two real roots, the value $b_{\Theta_2}$ needs to be negative, or equivalently, the integral $C(\ch_4)=\int_X \ch_4$ needs to be positive. For smooth hyperkähler fourfolds, this indeed holds by the bound of Guan (see~\cite[Lemma~4.6]{OberdieckSongVoisin} or~\cite[Theorem~7]{Sawon2021}).

\item When $n=3$, the graph homology class admits a factor $1+\frac{1}{12}\Theta$,
so we also get a factorization for the Riemann--Roch polynomial
\[
\RR_X(q)=\frac{C(1)}{(2\cdot 3)!}\left(q+\frac1{12}b_{\Theta}\right)
\left(q^2+\frac1{12}2b_{\Theta}q+\frac1{12^2}(b_{\Theta}^2+3b_{\Theta_2})\right).
\]
So if $b_{\Theta_2}$ is negative, the polynomial will indeed admit three real roots forming an arithmetic sequence, with difference $\frac1{12}\sqrt{-3b_{\Theta_2}}$.

%\sout{In general, maybe we could get a product of quadratic terms? (Plus a linear one if $n$ is odd.) All centered around $1+\frac1{24}t\Theta$. So in particular there should only be $\Theta$ and $\Theta_2$.} This is probably not true, because this leads to inconsistent predictions for the next term $\inner{\Omega^2,\ell^4}$...

\item When $n=4$, the graph homology class becomes more complicated due to the extra graph $\Xi$. If we expect the Riemann--Roch polynomial to admit four real roots forming an arithmetic sequence, this would lead to the following conjectural relations among certain generalized Fujiki constants.
\begin{conjecture}
\label{conj:Generalized_Fujiki_ch_8_via_RW}
If $X$ is of dimension $2n\geq 8$, then
\[
\frac{C(\ch_4^2+120\ch_8)\cdot C(1)}{C(\ch_4)^2}
=\frac{(5n+7)(2n-1)(2n-3)}{5(n+1)(2n-5)(2n-7)}.
\]
\end{conjecture}
Admitting this relation, we would then get
\[
\inner{\ch_4^2+120\ch_8,(2\sigma)^4}_\sigma=\big(\tfrac{5}{3(n+1)}+\tfrac{25}6\big)\RW_\sigma(\Theta_2^2).
\]
On the other hand, based on the computation of Sawon~\cite{SawonThesis}, we have
\[
\begin{aligned}
\tfrac1{384}\inner{w_4^2,\ell^4}&=24\Xi+48\Theta_4+\tfrac{25}4\Theta_2^2,\\
\tfrac1{384}\inner{w_8,\ell^4}&=7\Xi+\tfrac{287}8\Theta_4.
\end{aligned}
\]
Taking a suitable linear combination and applying $\RW_\sigma$, we get
\[
\inner{\ch_4^2+120\ch_8,(2\sigma)^4}_\sigma=\RW_\sigma(8\Xi-9\Theta_4+\tfrac{25}6\Theta_2^2),
\]
so
\[
\RW_\sigma(8\Xi-9\Theta_4) = \tfrac{5}{3(n+1)}\RW_\sigma(\Theta_2^2).
\]
Hence we can express the Rozansky--Witten invariant of
$\frac{144}{25}\Xi-\frac{162}{25}\Theta_4=\frac{18}{25}(8\Xi-9\Theta_4)$ in terms
of $b_{\Theta_2}$, so the Riemann--Roch polynomial has the following form
\[
\begin{aligned}
\RR_X&(q) =\\ 
&\frac{C(1)}{(2\cdot 4)!}\left(q^2+\tfrac1{12}2b_\Theta q+\tfrac1{12^2}(b_\Theta^2+\tfrac35
b_{\Theta_2})\right)
\left(q^2+\tfrac1{12}2b_\Theta q+\tfrac1{12^2}(b_\Theta^2+\tfrac{27}5
b_{\Theta_2})\right).
\end{aligned}
\]
%(This only holds {\em after} we apply $\RW$; also the relation depends on $n$.)
%We get
%\[
%\begin{aligned}
%\RR_X(q)&=\frac{C(1)}{(2\cdot 4)!}\left(q^4+\tfrac1{24}4a_1q^3
%+ \tfrac1{24^2}6(a_1^2+a_2)q^2 + \tfrac1{24^3}4a_1(a_1^2+3a_2)q + \tfrac1{24^4}(a_1^4+6a_1^2a_2+\tfrac{81}{25}a_2^2)
%\right)\\
%&=
%\frac{C(1)}{(2\cdot 4)!}\left(q^2+\tfrac1{24}2a_1q+\tfrac1{24^2}(a_1^2+\tfrac35
%a_2)\right)
%\left(q^2+\tfrac1{24}2a_1q+\tfrac1{24^2}(a_1^2+\tfrac{27}5
%a_2)\right)
%\end{aligned}
%\]
If $b_{\Theta_2}$ is negative, then it indeed admits four roots forming an arithmetic progression with difference $\frac16\sqrt{-\frac35b_{\Theta_2}}$.

%we may write $\delta=\sqrt{-\frac35a_2}$, and
%the four roots are
%\[
%-\tfrac1{24}(a_1+3\delta),\quad
%-\tfrac1{24}(a_1+\delta),\quad
%-\tfrac1{24}(a_1-\delta),\quad
%-\tfrac1{24}(a_1-3\delta).
%\]

\end{itemize}
\subsection{Conjectural value for generalized Fujiki constants}
\label{subsec:Generalized:Fujiki_Characteristic_classes}
In the above examples we see that the value $b_{\Theta_2}$ or equivalently $C(\ch_4)$ governs the differences between the roots of the Riemann--Roch polynomial. We speculate that the roots always form an arithmetic progression with difference~2. This is our main motivation for Conjecture~\ref{conj:Mercedes_graph}. Note that the conjectural value for $C(\ch_4)$ also predicts that one should always have $b_{\Theta_2}=-48(n+1)$ by Proposition~\ref{prop:RW_Theta2}.
It can also be seen as a weaker version of Conjecture~\ref{conj:general_form_RR_polynomial} (2), for purely algebraic reasons.
\begin{proposition}
\label{prop:conj_RR_poly_implies_conj_C(ch_4)}
Conjecture~\ref{conj:general_form_RR_polynomial} (2) implies Conjecture~\ref{conj:Mercedes_graph}. 
\end{proposition}
\begin{proof}
By assumption, the roots of $\RR_X(q)$ form an arithmetic progression with difference 2, so we have
\begin{align*}
    \RR_X(q) &= \frac{C(1)}{(2n)!}(q+a)(q+a+2) \cdots (q+a+2n-2)\\
    &=\frac{C(1)}{(2n)!}\Big(q^n + (na+n(n-1)) q^{n-1} \\
    &\qquad + \Big(\tfrac{n(n-1)}2a^2 + (n-1)^2na +\tfrac{(3n-1)n(n-1)(n-2)}6 \Big)q^{n-2} + \dots \Big)
\end{align*}
Then by the result of Corollary~\ref{cor:first_four_C}, we may deduce the values for $C(c_2^2)$ and $C(c_4)$, and consequently $C(\ch_4)$, which turns out to depend only on $C(1)$ and $n$, and not on $a$.
\begin{comment}
Moreover, we have
\begin{align*}
\inner{\Omega^2,\ell} &= \frac1{12}\Theta,\\
\inner{\Omega^2,\ell^2} &= \frac1{12^2}(\Theta^2+\Theta_2).
\end{align*}
The operator $\inner{\ell, \_}$ acts by results of Jiang as $2\Lambda_\sigma$. Thus, the above gives us the value of $\RW_\sigma(\Theta)$ in terms of $n$ and $a$. Using the same strategy and applying Proposition~\ref{prop:RW_Theta2} yields the assertion. 
\end{comment}
\end{proof}

We also explore some consequences of Conjecture~\ref{conj:Mercedes_graph}.
\begin{proposition}
\label{prop:possible_Fujiki_dim4}
Assuming Conjecture~\ref{conj:Mercedes_graph}, for $n=2$ the following are the only
possibilities for the generalized Fujiki constants of a hyperkähler fourfold.
\[
\begin{array}{|c|c|c|c|}
\hline
C(1) & C(c_2) & C(c_2^2) & C(c_4)\\
\hline
3 & 30 & 828 & 324\\
9 & 54 & 756 & 108\\
\hline
\end{array}
\]
\end{proposition}
\begin{proof}
We have the following three relations
\[
\begin{aligned}
7C(c_2^2) - 4C(c_4) &= 15\tfrac{C(c_2)^2}{C(1)},\\
C(c_2^2) - 2C(c_4) &= 60C(1),\\
3C(c_2^2) - C(c_4) &= 2160,
\end{aligned}
\]
from which we may deduce that
\[
\begin{aligned}
C(c_2)&=2\sqrt{C(1)^2+72C(1)},\\
C(c_2^2)&=-12C(1)+864,\\
C(c_4)&=-36C(1)+432.
\end{aligned}
\]
The top-degree ones are just Chern numbers, and using the relations on Betti
numbers by Salamon, we have
\[
c_2^2=736+4b_2(X)-b_3(X),\quad
c_4=48+12b_2(X)-3b_3(X).
\]
Since $b_3(X)$ is a multiple of 4, the Chern number $c_2^2$ must also be a
multiple of 4, so we have $C(1)\in\frac1{3}\bZ$. By the bounds of Guan, we
have $-120\le c_4\le 324$, hence $\frac{46}3\ge C(1)\ge 3$, so we only have a
finite number of possibilities left.

By definition, the generalized Fujiki constant $C(c_2)$ should be
rational. Using this property we may verify that only the listed two cases are
possible, which are realized by $\mathrm K3^{[2]}$ and $\Kum_2$ respectively.
%
\begin{comment}
{\tiny
(I used {\tt Sage} to make sure that I didn't make any silly arithmetic
mistake.)
\begin{verbatim}
[(c,2*sqrt(c^2+72*c),864-12*c,432-36*c) for c in srange(3,47/3,1/3) if (c^2+72*c).is_square()]
\end{verbatim}
}
\end{comment}
\end{proof}
We note that the above argument shows that for $n=2$, Conjecture~\ref{conj:general_form_RR_polynomial} (2) and Conjecture~\ref{conj:Mercedes_graph} are equivalent. 
And this further reduces the number of possibilities for Betti numbers to 4, as stated in the introduction.

\begin{proof}[Proof of Corollary~\ref{cor:possible_Hodge_numbers_via_conjecture}]
From Salamon's relations \cite{Salamon} one obtains the formula
\[
c_4=48+12b_2(X)-3b_3(X).
\]
By Proposition~\ref{prop:possible_Fujiki_dim4} there are only two possible values for $c_4$ which together with previously obtained bounds from Guan yield the assertion. 
\end{proof}

\begin{comment}
If Conjecture~\ref{conj:Mercedes_graph} is shown to be true, there are essentially only four types possible for a hyperkähler fourfold. 
Ultimately, one is lead to formulate the following. 
\begin{conjecture}
Let $X$ be a hyperkähler fourfold. Then $X$ is of $\mathrm{K3}^{[2]}$ or $\mathrm{Kum}_2$-type.
\end{conjecture}
\end{comment}

\bigskip

Finally, motivated by the degree 4 case, we conjecture the following behavior to be true for arbitrary dimensions. The question was also asked independently in \cite{CaoOberdieckToda}.
\begin{conjecture}
\label{conj:generalized_fujiki_Chern_character}
For $k_1, \dots, k_r \in \bZ_{>0}$ with $k \coloneqq\sum_i k_i\leq n$ we have
\[
(-1)^{k}C(\ch_{2k_1} \cdots \ch_{2k_r})>0 \quad \text{ as well as } \quad C(c_{2k_1} \cdots c_{2k_r})>0.
\]
\end{conjecture}
This in particular generalizes the conjectures in \cite[Questions~4.7 and~4.8]{OberdieckSongVoisin} to products which do not necessarily live in top degree.

The conjectured alternating behaviour of products of Chern characters together with the positivity of products of Chern classes would yield in combination many restrictions and inequalities between these characteristic values. We expect the above positivity to hold pointwise and to be of local nature. 

\appendix
\section{Singularities of the orbifold $K_4'$}
\label{app:K4}
In this appendix, we analyze the singularities of the orbifold denoted by
$K_4'$ in~\cite[Section~5.4]{FuMenet}, which has $b_2=6$.
The analysis in {\it loc. cit.} contains an error, so the numbers $a_m$ of singular points
    obtained (namely, $a_2=45$ and $a_4=2$) are not correct.
In fact, using these numbers one would get the following relation
\[
C(c_2)=\sqrt{142 C(1)},
\]
which would imply that the Fujiki constant $C(1)$ contains
the large prime factor $71$ and this seems very unlikely.

We first recall the construction of the orbifold $K_4'$.
Consider the elliptic curve $E\coloneqq\bC/\inner{1,i}$ and let $T\coloneqq
    E\times E$ be the product of two copies of $E$, which can also be identified as
    $\bC^2/\Lambda$, where $\Lambda$ is the lattice $\inner{(1,0), (i,0), (0,1),
    (0,i)}$.
Consider the symplectic automorphism $\sigma_4$ of order 4 on $T$
    given by the matrix
\[
    \sigma_4 = \begin{pmatrix}0&-1\\1&0\end{pmatrix}.
\]
In other words, $\sigma_4$ acts via $(a,b)\mapsto (-b,a)$.
We get an induced automorphism $\sigma\coloneqq\sigma_4^{[2]}$ on the
generalized Kummer fourfold $\Kum_2(T)\subset T^{[3]}$, which by definition is the variety parametrizing length-$3$ subschemes $\set{x,y,z}\subset T$ such that $x+y+z=0$.
Since $\sigma^2=(-\Id)^{[2]}$ is a symplectic involution, the fix locus of
$\sigma^2$ consists of 36 isolated points---$35=\frac13\binom{15}2$ of them are of the form $\set{x,y,-x-y}$
for $x,y\in T[2]\setminus\set0$, and the last one is the vertex that is
supported at $\set{0,0,0}$---as well as a surface
\[
\Sigma\coloneqq\overline{\setmid{\xi\in\Kum_2(T)}{\supp \xi=\set{0,-x,x}\text{
for }x\in  T\setminus\set{0}}},
\]
that is isomorphic to the Kummer surface $\Kum(A)$.
If we take the quotient of $\Bl_\Sigma\Kum_2(A)$ by the induced action of
$\sigma^2$, we would obtain an example of type $K'$ which has $b_2=8$.
The orbifold $K_4'$ is obtained as the further $\bZ/2$-quotient by the induced
action of $\sigma$.

We analyze the number of quotient singularities. The number of order-4
singularities is equal to the number of fixed points by $\sigma$ among the 36
points.
The vertex is clearly fixed by $\sigma_4$, since the action is given by
$(a,b)\mapsto(-b,a)$.

We now consider the 2-torsion points. There are three non-trivial 2-torsion
points on the elliptic curve $E$, which we denote by $s=\frac 12,t=\frac i2,u=\frac{1+i}2$.
They also satisfy the extra relation $s+t+u=0$.
In \cite{FuMenet}, it is claimed that $\set{(s,s),(t,t),(u,u)}$ is the only
fixed point by $\sigma$.
This is not true: the following two types of points are also fixed
\[
\set{(s,t),(t,s),(u,u)},\set{(s,0),(0,s),(s,s)},
\]
which by symmetry provide us with 6 more fixed points.
In this way we get $a_4=1+1+6=8$ while $a_2=(36-a_4)/2+(R-a_4)$, where $R$ is the
number of ramification points.
Then we follow the same computation as in \cite{FuMenet}: realizing $K_4'$ as a
$\bZ/2$-quotient of $K'$, we get
\[
2\chi(K_4')-R=\chi(K')=108\quad\Rightarrow\quad\chi(K_4')=(108+R)/2,
\]
while on the other hand,
\[
\begin{aligned}
\chi(K_4')&=b_4+2b_2+2\\
&=48+12b_2+s\quad\text{(orbifold Salamon relation,
\cite[Proposition~3.6]{FuMenet})}\\
&=48+12\cdot 6+(a_2\cdot(-1)+a_4\cdot(-3))\quad\text{(\cite[Equation~(16)]{FuMenet})},
\end{aligned}
\]
from which we may solve that
\[
R=24,\quad a_2=30,\quad\chi(K_4')=66.
\]
Thus, we can proceed to compute
\[
c_4=\chi-a_2\cdot\tfrac12-a_4\cdot \tfrac34=45
\]
and
\[
\td_4=\tfrac1{240}c_2^2-\tfrac1{720}c_4=3-a_2\cdot \tfrac1{32}-a_4\cdot
\tfrac9{64}=\tfrac{15}{16},
\]
so that we have
\[
c_2^2=(\tfrac{15}{16}+\tfrac{45}{720})\cdot 240=240.
\]
We summarize the results.
\begin{proposition}
For the orbifold $K_4'$, we have the following numerical invariants
\[
\begin{gathered}
a_2=30,\quad a_4=8,\quad \chi_\top(K_4')=66,\\
c_2^2=240,\quad c_4=45.
\end{gathered}
\]
\end{proposition}

To finish, we can use the
Hitchin--Sawon formula to obtain
\[
\frac{C(c_2)^2}{C(1)}=\tfrac1{15}(7c_2^2-4c_4)=100,
\]
so
\[
C(c_2)=10\sqrt{C(1)},
\]
which indeed no longer has the prime factor 71.

Again, one may observe that the Hitchin--Sawon formula poses strong arithmetic
constraints on the generalized Fujiki constants (and it can be used as a
quite effective tool to detect errors in computations).

\bibliography{pub_bib}
\end{document}